\documentclass[11pt]{article}
\usepackage{amsmath,amsfonts,amssymb,graphicx,amsthm,color,hyperref}
\usepackage{authblk}
\usepackage[normalem]{ulem}
\allowdisplaybreaks
\usepackage{geometry}

\usepackage[title]{appendix}
\newcommand{\E}{\operatorname{E}}
\newcommand{\cov}{\operatorname{Cov}}
\newcommand{\var}{\operatorname{Var}}

\newtheorem{theorem}{Theorem}[]

\newtheorem{proposition}[]{Proposition}
\newtheorem{corollary}[]{Corollary}
\newtheorem{definition}[]{Definition}
\newtheorem{assumption}[]{Assumption}

\theoremstyle{definition}
\newtheorem{remark}[]{Remark}
\newtheorem{example}[]{Example}
\usepackage{natbib}

\begin{document}
\title{Likelihood distortion and Bayesian local robustness}
\author[1,2\footnote{antonio.dinoia(\sout{at})stat.math.ethz.ch}]{{Antonio} {Di Noia}}
\author[3]{{Fabrizio} {Ruggeri}}
\author[2,4]{{Antonietta} {Mira}}
\affil[1]{Seminar for Statistics, Department of Mathematics, ETH Zurich}
\affil[2]{Faculty of Economics, Euler Institute, Università della Svizzera italiana}
\affil[3]{Institute for Applied Mathematics and Information Technologies, National Research Council of Italy}
\affil[4]{Department of Science and High Technology, University of Insubria}
\date{\today}
\setcounter{Maxaffil}{0}
\renewcommand\Affilfont{\itshape\small}
\maketitle
\begin{abstract}
Robust Bayesian analysis has been mainly devoted to detecting and measuring robustness w.r.t.\ the prior distribution. Many contributions in the literature aim to define suitable classes of priors which allow the computation of variations of quantities of interest while the prior changes within those classes. The literature has devoted much less attention to the robustness of Bayesian methods w.r.t.\ the likelihood function due to mathematical and computational complexity, and because it is often arguably considered a more objective choice compared to the prior. 
In this contribution, we propose a new approach to Bayesian local robustness, mainly focusing on robustness w.r.t.\ the likelihood function. Successively, we extend it to account for robustness w.r.t.\ the prior, as well as the prior and the likelihood jointly.
This approach is based on the notion of distortion function introduced in the literature on risk theory. The novel robustness measure is a local sensitivity measure that turns out to be very tractable and easy to compute for several classes of distortion functions. Asymptotic properties are derived, and numerical experiments illustrate the theory and its applicability for modelling purposes.
\end{abstract}


\section{Introduction}
\label{sec:intro}
\subsection{Background and contributions}
Robustness has been defined as the insensitivity of statistical procedures to small deviations from assumptions \citep{huber2011robust} and, specifically, robust Bayesian analysis aims to detect and measure the
uncertainty and sensitivity induced by the common modelling choices necessary in Bayesian statistics, i.e. the prior, the likelihood and the loss function. The most explored problem in the literature is related to the definition of suitable classes of priors which allow the derivation of ranges for some quantities of interest while the prior changes in those classes; this approach is called global sensitivity or global robustness analysis. Some of those classes are defined to be neighbourhoods of the elicited prior and might define a topology on the space of probability measures. Among them, we could mention the $\epsilon$-contamination class \citep{berger86}, the mixture class \citep{gelfand1991bayesian} also known as the geometric $\epsilon$ -contamination class, the density ratio class \citep{deroberts1981bayesian, wasserman1992invariance}, and the classes based on concentration functions \citep{Fortini2000}. Other classes have been defined through conditions on the marginal distribution of the data; see \cite{betro1994}.
Moreover, \cite{stanca2023robust} recently introduced a choice-theoretic axiomatisation of the prior robustness problem.
An alternative approach is based on measuring the impact of infinitesimal perturbations: this is called local sensitivity or local robustness analysis; see e.g.\ \cite{gustafson1995local}, \cite{Ruggeri93} and \cite{ruggeri1995density}. For some proposals in the case of robustness w.r.t.\ the loss function, see \cite{makov}, \cite{martin98}, \cite{martin2003joint} and references therein.

The literature has devoted much less attention to Bayesian robustness w.r.t\ the likelihood. Some proposals are presented in \cite{dey1996local} and \cite{kurtek2015bayesian}, however, although the likelihood is often considered a more objective choice compared to prior, we believe that one of the major reasons why the literature lacks such a methodological tool relies on mathematical intractability.

The main contribution of this work is the proposal of a new approach to Bayesian robustness w.r.t.\ the likelihood function, which can also be easily extended to measure robustness w.r.t.\ the prior and to both. This approach is based on the notion of distortion function that has been proposed in the literature on risk theory, mathematical finance, and actuarial science; see, e.g.\ \cite{wang_1996} and \cite{yaari}. In the context of Bayesian robustness, distortion functions have been successfully employed in \cite{distorsion_priors2016} and \cite{ruggeri2021new} to build the so-called distorted bands of priors, which allow a simple way to measure global robustness w.r.t.\ the prior. Here we propose a local sensitivity measure that is very easy to compute under certain classes of distortion functions.

\subsection{Notation and setting}
Let $(\mathcal{X},\mathcal{F},P)$ and $(\Theta,\mathcal{B},\Pi)$ be, respectively, the sample probability space and the parameter probability space, and let us assume to model random experiments by adopting a statistical model $(P_\theta: \theta \in \Theta)$ for the population distribution. 
Let $X:\mathcal{X}\to \mathbb{R}^d$ be a random vector (r.v.) with distribution $P_\theta$, and let $F_X$, $S_X$ and $f_X$ be respectively its distribution, survival, and density (assuming it admits a version) functions parameterized by $\theta\in \Theta \subseteq \mathbb{R}^k$. We explicitly work with absolutely continuous distributions since it makes the notation and mathematical development easier; however, we also show that all the results hold for discrete distributions as well. 
Let $\theta$ be the realisation of a r.v.\ $\vartheta:\Theta \to \Theta$ with distribution $\Pi$, distribution, survival and density (assuming it admits a version) functions $F_\vartheta$, $S_\vartheta$ and $f_\vartheta$, respectively.  
Let us denote the Bayesian statistical model by the joint r.v.\ $(X,\vartheta)$ defined on the product measurable space $(\mathcal{X}\times \Theta,\mathcal{F}\otimes \mathcal{B} )$. We denote the Lebesgue measure on $\mathbb{R}^d$ or $\mathbb{R}^k$ by $\lambda$ and, for a measure $\mu$ on a given measurable space $(\Omega, \mathcal{A})$, let $L^p(\mu)= \{f: \Omega \to \mathbb{R}^d \,\, \mathrm{measurable} \,\,\mathrm{s.t.} \,\int_{\Omega} f(\omega)^p d\mu(\omega)<\infty\}$. Given a distribution $\Pi$ we denote the expectation, variance and covariance operators w.r.t.\ $\Pi$ by $\E_\Pi$, $\var_\Pi$ and $\cov_\Pi$, respectively.
Throughout the paper, unless differently specified, we consider an independent and identically distributed (i.i.d.) random sample $X_1,\dots,X_n$ where each r.v.\ has the same distribution of $X$, and we adopt the compact notation $X_{1:n}$ and $x_{1:n}$ for its realisation.

\subsection{Organisation of the paper}
The paper is organised as follows:
Section~\ref{sec:dist} introduces the distortion technique and a new class of local sensitivity measures w.r.t.\ the likelihood, the prior, and both.
In Section~\ref{sec:classes}, we provide a closed-form expression of the local sensitivity under power distortion, as well as upper and lower bounds;
in addition, extensions to other distortion functions are presented.
In Section~\ref{sec:asymptotic_behaviour}, some asymptotic properties of the proposed robustness measure are presented; in particular, we derive its asymptotic behaviour under posterior consistency. In Section~\ref{sec:num}, we work with both simulated and real data to verify the theoretical results and to illustrate the applicability of the proposed measure, providing promising insights on model choice and practical robustness analysis. Conclusions and future research directions are presented in Section~\ref{sec:conc}. The proofs of all theoretical results are provided in Appendix~\ref{sec:proofs}.

\section{The distortion method and local robustness}
\label{sec:dist}
Following \cite{distorsion_priors2016}, we recall the definition of distortion function.
\begin{definition}
    A function $h:[0,1]\to [0,1]$ is a distortion function if it is differentiable, monotone non-decreasing and such that $h(0)=0$ and $h(1)=1$.
\end{definition}
Under the setting introduced in Section~\ref{sec:intro}, let us consider the composition $F_{X_h} :=h \circ F_X$, which defines a new r.v.\ $X_h$ with distribution $P_{\theta,h}$, distribution function $F_{X_h}(x |  \theta)=h(F_X(x |  \theta))$ and density
\begin{align}
    \label{eq:cdf}
    f_{X_h}(x |  \theta)= \frac{\partial F_{X_h}}{\partial x}(x |  \theta) =  f_X(x |  \theta)h'(F_X(x |  \theta)).
\end{align}
Similarly to \eqref{eq:cdf}, if we compose the survival function, letting $S_{X_h}:=h \circ S_X$, it follows
\begin{align}
\label{eq:survival}
    \Tilde{f}_{X_h}(x |  \theta)= -\frac{\partial S_{X_h}}{\partial x}(x |  \theta) = -f_X(x |  \theta) (-h'(S_X(x |  \theta))).
\end{align}
Identical considerations can be made if we consider the composition $F_{\vartheta_h}=h\circ F_\vartheta$, obtaining a distortion of the prior with associated distribution denoted by $\Pi_h$. Moreover, it is worth noting that \eqref{eq:cdf} and \eqref{eq:survival} specify weighted versions of the original r.v.\ with weight function $w(x)= h'(F_X(x |  \theta)) $ or $w(x)= h'(S_X(x |  \theta)) $, in the same spirit as \cite{blazej2008preservation} and \cite{ruggeri2021new}.
We are now ready to introduce the distorted likelihood and the distorted posterior distribution. For ease of exposition, in the following definition, we directly use Bayes' rule to obtain a version of the distorted posterior distribution omitting the usual measurability conditions.
\begin{definition}
    Let $X_{1:n}$ be a random sample and let $h$ be a distortion function, we define its $h$-distorted likelihood function as
    \begin{align}
    \label{eq:lik}
        \mathcal{L}_h(\theta)=f_{X_{h,1:n}}(x_{1:n}|\theta)= f_{X_{1:n}}(x_{1:n}|\theta) h'(F_{X_{1:n}}(x_{1:n}|\theta)).
    \end{align}
    Moreover, (a version of) the distorted posterior distribution based on the $h$-distorted likelihood $\mathcal{L}_h$ and on the $h$-distorted prior distribution $\Pi_h$ is given by
    \begin{align*}
        \Pi^{n}_h(B):=\Pi_h(B| x_{1:n})= \frac{\int_B \mathcal{L}_h(\theta) d \Pi_h(\theta)}{\int_\Theta \mathcal{L}_h(\theta)d\Pi_h(\theta)}
    \end{align*}
    for every $B\in \mathcal{B}$. Moreover, for notation convenience, we will use $\Pi^n$ and $\Pi(\cdot | x_{1:n})$ interchangeably to denote the non-distorted posterior distribution.
\end{definition}
Our development will mainly focus on the case where the distorted posterior is obtained as a single distortion of the likelihood. Moreover, it is straightforward that when $X_{1:n}$ are independent copies of $X$, then \eqref{eq:lik} factorizes as
\begin{align*}
    \mathcal{L}_h(\theta)= \prod_{i=1}^n f_X(x_i |  \theta)h'(F_X(x_i |  \theta)).
\end{align*}
Note that the i.i.d.\ assumption is not strictly necessary; however, it strongly simplifies the notation and the simplicity of the exposition. The relaxation of such assumptions will be discussed successively since all the arguments can be easily extended to a fairly wide class of random processes.

\begin{remark}
\label{rmk:discrete}
When $X$ is a discrete r.v.\ with probability mass function $p_{X}$ and distribution function $F_X$, we obtain an equivalent expression for the distorted version of $p_X$.
Indeed, let $e_j$ be an element of the canonical basis of $\mathbb{R}^d$, since $h$ is differentiable, by the {Mean Value Theorem} there exists a {scalar} $z\in
(F_{X_{1:n}}(x_{1:n}| \theta), F_{X_{1:n}}(x_{1:n} + {e}_j | \theta))
$
for some $j\in \{1,\dots,d\}$, such that the distorted likelihood satisfies  
$
\mathcal{L}_h(\theta) = p_{X_{1:n}}(x_{1:n} | \theta) h'(z).
$
\end{remark}

\begin{example}
    Let $X\sim \mathrm{EF}(\theta)$ where $\mathrm{EF}(\theta)$ denotes the exponential family of distributions with $\theta \in \mathbb{R}^k$ and let also $\vartheta\sim \mathrm{EF}(\xi)$ with $\xi \in \mathbb{R}^q$. Moreover, let $h: [0,1] \to [0,1]$ be a distortion function, it is immediate to note that there exist positive functions $c: \mathbb{R}^k \to \mathbb{R}_+$ and $b: \mathbb{R}^d \to \mathbb{R}_+$ and, for $j=1,\dots,k$, real-valued functions $w_j:\mathbb{R}^k \to \mathbb{R}$ and $t_j: \mathbb{R}^d \to \mathbb{R}$ such that 
    $$\mathcal{L}_h(\theta)= \Big(\prod_{i=1}^n b(x_i)\Big) c(\theta)^n   \exp\Big( \sum_{j=1}^k (w_j(\theta) \sum_{i=1}^n t_j(x_i))\Big) \prod_{i=1}^n h'\Big(\int_{-\infty}^{x_i} f_X(y|\theta) d\lambda(y)\Big)$$
    where $f_X(x|\theta)=b(x)c(\theta) \exp (\sum_{j=1}^k w_j(\theta) t_j(x))$. Therefore, it should be noted that, in general, there is no guarantee that the distorted version $X_h$ belongs to the exponential family. However, the distorted posterior belongs to the exponential family, because the additional component appearing in the likelihood due to distortion given by the function $(x_{1:n},\theta)\mapsto \prod_{i=1}^n h'(\int_{-\infty}^{x_i} f_X(y|\theta) d\lambda(y)) $ can simply be considered a function of $\theta$, thus the exponential family property of $\vartheta_h|x_{1:n}$ is preserved.
\end{example}

\begin{example}
\label{ex:exponential-gamma1}
Let us consider the distortion function $h(y)=\Tilde {h}_\alpha(y)= 1-(1-y)^\alpha$ with $\alpha>1$.
Let $X\sim \mathrm{Exponential}(\theta)$, the $\Tilde{h}_\alpha$-distorted likelihood function based on an i.i.d.\ sample of size $n$ is
\begin{align}
\label{eq:dist_lik_ex}
        \mathcal{L}_{\Tilde{h}_\alpha}(\theta)
        &= \prod_{i=1}^n \theta e^{-\theta x_i}\alpha (1-F_X(x_i |  \theta)) ^{\alpha-1}=(\theta \alpha)^n \exp\Big(-\theta \alpha\sum_{i=1}^nx_i\Big).
\end{align}
Let $a, b >0$ and consider the prior $\vartheta\sim \mathrm{Gamma}(a,b)$, we have that (a version of) the $\Tilde{h}_\alpha$-distorted
posterior density is
 \begin{align*}
    \frac{d \Pi_{\Tilde{h}_\alpha}^{n} }{d\lambda}(\theta) &\propto (\theta\alpha)^n \exp\Big(-\theta \alpha\sum_{i=1}^nx_i\Big) \theta^{a-1}\exp(-b\theta) \\
     &\propto\theta^{a+n-1}\exp\Big(-\theta\Big(b+\alpha\sum_{i=1}^n x_i \Big)\Big),
 \end{align*}
i.e.\ the $\Tilde{h}_\alpha$-distorted
posterior distribution is $\mathrm{Gamma} (a+n,b+\alpha \sum_{i=1}^n x_i )$. Moreover, it is worth noting that the distorted posterior mean
$$\E_{\Pi_{\Tilde{h}_\alpha}^{n}}[\vartheta]= \frac{a+n}{b+\alpha \sum_{i=1}^n x_i}$$
is a decreasing function of $\alpha$, and the non-distorted version is obtained as $\alpha \downarrow 1$. This is a direct consequence of the adopted distortion which assigns more probability mass to small values of $X$ when $\alpha$ increases. Therefore, although the data has not changed, inspecting \eqref{eq:dist_lik_ex} it is at once apparent that the distortion has modified the original model from $X\sim \mathrm{Exponential}(\theta)$ to $X_{\Tilde{h}_\alpha}\sim \mathrm{Exponential}(\alpha \theta)$. At the same time, the Bayesian posterior mean point estimator has changed from $ \E_{\Pi^{n}}[\vartheta]= ({a+n})/({b+\sum_{i=1}^n x_i})$ to $ \E_{ \Pi_{\Tilde{h}_\alpha}^{n}}[\alpha \vartheta]= {\alpha (a+n)}/({b+\alpha \sum_{i=1}^n x_i}).$
Looking at their ratio, we have
\begin{align*}
\frac{\E_{\Pi_{\Tilde{h}_\alpha}^{n}}[\alpha \vartheta]}{\E_{ \Pi ^{n}}[\vartheta]}= \frac{\alpha (b+\sum_{i=1}^n x_i)}{b+\alpha \sum_{i=1}^n x_i}.
\end{align*}
Differentiating the ratio w.r.t.\ $\alpha$, it is evident how it increases in $\alpha$. This makes sense since the distortion function gives more and more weight to small values of $X$ as $\alpha$ increases, consequently implying a decreasing mean and an increasing parameter of the exponential distribution. We can also observe that $\E_{\Pi_{\Tilde{h}_\alpha}^{n}}[\alpha \vartheta]$ ranges from $0$ to $ (a+n)/(\sum_{i=1}^n x_i)$, the latter close to the maximum likelihood estimator, which is obtained for $a=0$.
\end{example}
The following results aim to propose a new measure of local sensitivity obtained by a distortion of the likelihood function. In Bayesian robustness, a measure of local sensitivity is used to detect whether very small perturbations of the prior or likelihood imply significant modifications to the results of the statistical procedure. For instance, this is done by studying the impact of such perturbations on posterior quantities of interest, e.g.\ the posterior mean.
On the other hand, global robustness analysis is performed by studying how such posterior quantities of interest change while the prior or the likelihood range in a certain class.
For instance, in Example~\ref{ex:exponential-gamma1} a Bayesian global robustness analysis might be carried out studying how $\E_{\Pi_{h_\alpha}^{n}}[\vartheta]$ changes as $\alpha$ ranges in a proper interval, while local robustness is carried out studying the impact on $\E_{\Pi_{h_\alpha}^{n}}[\vartheta]$ of an infinitesimal perturbation.

In the following, we provide a notion of Bayesian local sensitivity obtained through distortion functions. Although the definition is valid for local robustness w.r.t.\ both prior and likelihood, we will mainly focus on the latter.
\begin{definition}
\label{def:local_sensitivity}
Let $(X,\vartheta)$ be a Bayesian statistical model and let $h_\alpha$ be a distortion function parameterized by a constant $\alpha \in \mathbb{R}$. The Bayesian (posterior) local sensitivity is defined as
\begin{align}
    \label{eq:posterior_sensitivity}
   \delta^{h_\alpha}_n:= \frac{\partial}{\partial \alpha} \E_{\Pi_{h_\alpha}^n} [g(\vartheta)]\Big|_{\alpha=\alpha_0}
\end{align}
where $\alpha_0$ is such that $h_{\alpha_0}(y)=y$ and
 $g:\Theta \to \mathbb{R}^s$ is a measurable function.
\end{definition}

Note that whenever \eqref{eq:posterior_sensitivity} is defined for all continuous bounded $g$ we can characterise the weak topology of the parameter space so that it is not restrictive to focus on such posterior expectations. For instance, we can take $g(\vartheta)=\vartheta$ obtaining the (distorted) posterior mean, or $g(\vartheta)= \mathbf{1}_{C_{1-\gamma}(x_{1:n})}(\vartheta)$ where $C_{1-\gamma}(x_{1:n})$ is the $100\times (1-\gamma)\%$ credible set for $\theta$ based on the sample realization $X_{1:n}=x_{1:n}$, leading to the (distorted) posterior probability $\Pi^n_{h_\alpha}(C_{1-\gamma}(x_{1:n}))$.
Definition~\ref{def:local_sensitivity} provides a very interpretable and mathematically convenient notion of local sensitivity under appropriate smoothness conditions on the distorted posterior mean. 

An alternative definition of Bayesian local sensitivity based on likelihood distortion is to measure perturbations on the expectation of the posterior predictive distribution leading to different more involved local sensitivity measures. For the following definition, we consider the distorted posterior predictive distribution, that is $
\Bar{P}^{n}_{h_\alpha}(B) := \int_\Theta {P}_{\theta, h_\alpha}(B) \, d\Pi^n_{h_\alpha}(\theta)
$ for any Borel set $B \subseteq \mathbb{R}^d$.
\begin{definition}
\label{def:local-sensitivity-predictive}
Let $(X,\vartheta)$ be a Bayesian statistical model and let $h_\alpha$ be a distortion function parameterized by a constant $\alpha \in \mathbb{R}$. The (posterior predictive) Bayesian local sensitivity is defined as
    \begin{align}
    \label{eq:predictive_local_sens}
    \delta_{n,\mathrm{pred}}^{h_\alpha}:=\frac{\partial}{\partial \alpha} \E_{\Bar{P}^{n}_{h_\alpha}}[g(X)]\Big|_{\alpha=\alpha_0}
    \end{align}
    where $\alpha_0$ is such that $h_{\alpha_0}(y)=y$,
 $g:\mathbb{R}^d \to \mathbb{R}^s$ is a measurable function.
\end{definition}
Both Definitions~\ref{def:local_sensitivity}~and~\ref{def:local-sensitivity-predictive} could be extended to take into account local curvature; however, this goes beyond the scope of the paper and is left for future research.
Now, following Definition~\ref{def:local_sensitivity}, we provide a first result regarding the variation of posterior expectations as $\alpha$ changes; in particular, in the next theorem, we provide a simple expression for $\delta^{h_\alpha}_n$ without assuming any specific functional form for the distortion function.
We will require the following technical condition on $h_\alpha$.
\begin{assumption}
\label{assump:distortion}
    The function $\alpha \mapsto \log \circ \, h'_\alpha$ is $\Pi$-almost surely differentiable at $\alpha=\alpha_0$.
\end{assumption}

Note also that, in the next results, we assume that the sample is fixed, while this assumption will be removed in the asymptotic setting of Section~\ref{sec:asymptotic_behaviour}.
\begin{theorem}
\label{thm:lik_dist}
Let $h_\alpha$, with $\alpha \in \mathbb{R}$, be a distortion function composing $F_X$ and satisfying Assumption~\ref{assump:distortion}. Let $\delta_0>0$ and assume that, for some $g$ measurable and $G\in L^1(\Pi)$, 
$$
| \delta^{-1}{g(\theta)(\mathcal{L}_{h_{\alpha_0+\delta}}(\theta)-\mathcal{L}_{h_{\alpha_0}}(\theta))} | \leq G(\theta,\alpha_0 ), \quad |\delta|\leq \delta_0
$$
holds $\Pi$-almost surely, then
\begin{align}
\label{eq:general_cov_lik}
  \delta^{h_\alpha}_n=  \cov_{\Pi^n} \Big[g(\vartheta), \sum_{i=1}^n \frac{\partial}{\partial \alpha}  \log  h'_\alpha (F_X(x_i |  \vartheta))\Big|_{\alpha=\alpha_0}\Big].
\end{align}
\end{theorem}

Remarkably, since the covariance is taken w.r.t.\ the non-distorted posterior $\Pi^n$, straightforward posterior approximations are available as we do not require to sample from the distorted posterior distribution. To interpret the result obtained, for the sake of simplicity, we may consider a sample composed of one observation $X_1$, then it is immediate to note that \eqref{eq:general_cov_lik} measures the linear relationship between each component of $g(\vartheta)$ and the (random) relative slope of the function $\alpha \mapsto h'_\alpha (F_X(x_i |  \vartheta))$ evaluated at $\alpha_0$. Furthermore, recall that $h'_\alpha (F_X(x_i |  \vartheta))$ is also the multiplicative component that appears in the distorted density given in \eqref{eq:cdf}.

It should be pointed out, throughout the paper, that since $g:\mathbb{R}^k \to \mathbb{R}^s$ 
it follows that $\delta^{h_\alpha}_n$ is a vector in $\mathbb{R}^s$. Therefore, the local sensitivity should be read as an element-wise covariance. It would be possible to summarise the information by considering its Euclidean norm, however, in our development, we will take the local sensitivity as a vector, because it is mathematically more convenient and, practically, it preserves information on the marginal local sensitivity for the components of $\vartheta$.
To clarify, if we consider $g(\vartheta)=\vartheta_j$ where $\vartheta_j$ is the $j$-th element of $\vartheta$, we obtain
\begin{align*}
    \delta^{h_\alpha}_n=\cov_{\Pi^n} \Big[\vartheta_j,\sum_{i=1}^n\frac{\partial}{\partial \alpha} \log  h'_\alpha (F_X(x_i |  \vartheta))\Big|_{\alpha=\alpha_0} \Big].
\end{align*}
 Now, we get the following bound by a straightforward application of the Cauchy-Schwarz inequality.
\begin{corollary}
\label{cor:inequality}
    Under the assumptions of Theorem~\ref{thm:lik_dist} it holds
    $$|\delta^{h_\alpha}_n|   \leq \var_{\Pi^n} [g(\vartheta)]^{1/2} \var_{\Pi^n}  \Big[\sum_{i=1}^n \frac{\partial}{\partial \alpha}  \log  h'_\alpha (F_X(x_i |  \vartheta))\Big|_{\alpha=\alpha_0} \Big]^{1/2}.$$
\end{corollary}

\begin{remark}
    Consider $X_{1:n}$ being non i.i.d.\ r.v.s, and assume that the r.v.s $X_i |  X_{-i}$ for $i=1,\dots,n$, are independent. It suffices to note that the following factorization holds
    \begin{align*}
        f_{X_{h,1:n}}(x_{1:n} |  \theta)= \prod_{i=1}^n f_{X_i |  X_{-i}}(x_i |  x_{-i},\theta)h'(F_{X_i |  X_{-i}}(x_i |  x_{-i},\theta))
    \end{align*}
    giving rise to a proper standard likelihood function. This leads exactly to the same results for the local sensitivity, i.e.\ we obtain
    \begin{align*}
    \delta^{h_\alpha}_n  = \cov_{\Pi^n}\Big[g(\vartheta),  \sum_{i=1}^n \frac{\partial}{\partial \alpha}  \log  h'_\alpha (F_{X_i |  X_{-i}}(x_i |  x_{-i},\vartheta))\Big|_{\alpha=\alpha_0} \Big].
    \end{align*}
\end{remark}

Two further results that we state below involve distortion of the prior and double distortion of both prior and likelihood. This leads us to the following Theorems.
\begin{theorem}
\label{thm:prior_distortion}
Let $h_\alpha$, with $\alpha \in \mathbb{R}$, be a distortion function composing $F_\vartheta$ and satisfying Assumption~\ref{assump:distortion}. 
Let $\delta_0>0$ and assume that, for some $g$ measurable and $G\in L^1(\Pi)$,
$$ |\delta^{-1} g(\theta)\mathcal{L}(\theta)(h'_{\alpha_0+\delta}(F_\vartheta(\theta))-h'_{\alpha_0}(F_\vartheta(\theta))) | \leq G(\theta,\alpha_0), \quad |\delta|\leq \delta_0$$ holds $\Pi$-almost surely, then
\begin{align}
\label{eq:general_cov_prior}
   \delta^{h_\alpha}_n=\cov_{\Pi^n} \Big[g(\vartheta),  \frac{\partial}{\partial \alpha}  \log  h'_\alpha (F_\vartheta(\vartheta))\Big|_{\alpha=\alpha_0} \Big].
\end{align}
\end{theorem}

To conclude the section, we state the result for the local sensitivity measure based on a double distortion of both the prior and the likelihood.

\begin{theorem}
\label{thm:double_distortion}
Let $h_\alpha$, with $\alpha \in \mathbb{R}$, be a distortion function composing both $F_\vartheta$ and $F_X$ and satisfying Assumption~\ref{assump:distortion}. 
Let $\delta_0>0$ and assume that, for some $g$ measurable and $G\in L^1(\Pi)$,
$$ | \delta^{-1}g(\theta)(\mathcal{L}_{h_{\alpha_0+\delta}}(\theta)h'_{\alpha_0+\delta}(F_\vartheta(\theta))-\mathcal{L}_{h_{\alpha_0}}(\theta)h'_{\alpha_0}(F_\vartheta(\theta))) | \leq G(\theta,\alpha_0), \quad |\delta|\leq \delta_0$$ holds $\Pi$-almost surely, then
\begin{align}
\label{eq:general_cov_double}
  \delta^{h_\alpha}_n=\cov_{\Pi^n}\Big[g(\vartheta),  \frac{\partial}{\partial \alpha}  \Big(\log  h'_\alpha (F_\vartheta(\vartheta))+ \sum_{i=1}^n \log  h'_\alpha (F_X(x_i |  \vartheta)) \Big)\Big|_{\alpha=\alpha_0} \Big].
\end{align}
\end{theorem}

\begin{remark}
\label{rmk:curvature}
   Definition~\ref{def:local_sensitivity} does not take into account the local curvature of the posterior, which would require the calculation of the second-order derivative. Nevertheless, Theorems~\ref{thm:lik_dist},~\ref{thm:prior_distortion}~and~\ref{thm:double_distortion} can be extended to account for local curvature under slightly modified conditions. However, the introduction of second-order properties goes beyond the scope of the paper, and such extensions are left for future research.
\end{remark}

\section{Local sensitivity measures}
\label{sec:classes}

Lemma 1 in \cite{distorsion_priors2016} shows that selecting a couple of convex and concave distortion functions induces a proper neighbourhood of the original distribution in terms of the likelihood ratio order. The authors refer to this type of neighbourhood as the distorted band class. Moreover, in Theorem~8 the authors also show that the distorted band class is a subset of the concentration function class that, in turn, is shown to be a topological neighbourhood by \cite{fortini1994defining}. In the following, we specify the results of Section~\ref{sec:dist} for some selected classes of distortion functions.
Note that any distortion function satisfying Assumption~\ref{assump:distortion} can be chosen, however, the elicitation of a specific distortion function should be driven by the specific robustness diagnostic that is of interest in a certain modelling setting.

\subsection{Local sensitivity under power distortion}

Among the available distortion functions, power distortion provides a simple and mathematically convenient one.
Moreover, its practical relevance motivates us to illustrate the theory in full detail.
\begin{definition}[Power distortion]
    Let $\alpha>1 $, a power distortion is induced by the functions $h_\alpha:[0,1]\to [0,1]$ and $\Tilde{h}_\alpha:[0,1]\to [0,1]$ given by
    \begin{align}
        \label{eq:pow_dist}
        y\mapsto h_\alpha(y)=y^\alpha, \quad \quad  y\mapsto\Tilde{h}_\alpha(y)=1-(1-y)^\alpha.
    \end{align}
\end{definition}
Note that $h_\alpha$ is strictly convex on $[0,1]$ and $\Tilde{h}_\alpha$ is strictly concave on $[0,1]$, thus, power distortion induces a proper neighbourhood in terms of likelihood ratio order. This is particularly useful in the context of Bayesian global robustness where we are interested in computing the so-called range measure. However, it is also important in the context of Bayesian local robustness, since it allows us to measure the impact of infinitesimal perturbations towards likelihood ratio neighbourhoods.
It is also interesting to note that when $\alpha:=n \in \mathbb{N}$, the power distortion functions in \eqref{eq:pow_dist} lead to the distribution of the order statistics $X_{(n)}$ and $X_{(1)}$, respectively. Thus, from a practical modelling perspective, it provides a manageable expression for detecting sensitivity to small perturbations of the tails and for measuring robustness to outliers.

Following Definition~\ref{def:local_sensitivity}, we adopt power distortion to perform a Bayesian local sensitivity analysis. In the following, we show that adopting a power distortion leads to a local sensitivity measure that is very easy to compute in practice.
\begin{proposition}
\label{prop:power_distorsion}
Let $h_\alpha(y)=y^\alpha $ with $\alpha>1$, under the assumptions of Theorem~\ref{thm:lik_dist} it holds
\begin{align}
\label{eq:power_distortion_lik}
   \delta^{h_\alpha}_n = \cov_{\Pi^n}\Big[g(\vartheta) ,\sum_{i=1}^n \log F_X(x_i |  \vartheta)\Big].
\end{align}
\end{proposition}

Note that the local sensitivity $\delta^{h_\alpha}_n$ under power distortion measures the linear relationship between the r.v.s $g(\vartheta)$ and $\sum_{i=1}^n \log F_X(x_i |  \vartheta)$ and, as a consequence, a stronger linear relationship implies a likelihood choice which is less robust to small modifications of the tails.
It is worth noting that under the same assumptions, owing to Corollary~\ref{cor:inequality}, we obtain the following bound
\begin{align}
\label{eq:bound-CS-distr}
|\delta^{h_\alpha}_n |\leq \var_{\Pi^n}[g(\vartheta)]^{1/2}  \var_{\Pi^n} \Big[\sum_{i=1}^n\log F_X(x_i |  \vartheta)\Big]^{1/2},
\end{align}
from which it immediately follows that \eqref{eq:power_distortion_lik} exists and it is bounded when the r.v.s $g(\vartheta)$ and $\sum_{i=1}^n \log F_X(x_i |  \vartheta)$ have finite second moment. Both conditions are generally fulfilled giving sharp bounds when $F_X$ is such that $\Pi^n(\theta : F_X(x_i|\theta)>0)=1$ for all $i=1,\dots,n$. 

Proposition~\ref{prop:power_distorsion} suggests the introduction of a simple local sensitivity index that measures the robustness of the model to infinitesimal perturbations of the tails, given the realisation of a random sample $X_{1:n}$. 
It is important to notice that in applied settings it might be useful to normalize $\delta^{h_\alpha}_n$ for its upper bound obtaining
\begin{align}
\label{eq:delta_lik_normalized}
    \overline{\delta}^{h_\alpha}_n:= \delta^{h_\alpha}_n\var_{\Pi^n}[g(\vartheta)]^{-1/2}  \var_{\Pi^n} \Big[\sum_{i=1}^n\log F_X(x_i |  \vartheta)\Big]^{-1/2}
\end{align}
which ranges in $[-1,1]$.

Although not the main purpose of the paper, it is worth mentioning that a result similar to Proposition~\ref{prop:power_distorsion} can be obtained for local sensitivity w.r.t.\ the prior as stated in Theorem~\ref{thm:prior_distortion}. Indeed, thanks to \eqref{eq:general_cov_prior} we obtain
\begin{align}
\label{eq:delta_prior}
     \delta^{h_\alpha}_n= \cov_{\Pi^n}[g(\vartheta), \log F_\vartheta(\vartheta)]
\end{align}
where $F_\vartheta$ is the distribution function of the prior distribution.
In addition, following Theorem~\ref{thm:double_distortion} we could obtain a sensitivity measure based on a double distortion of both the prior and the likelihood. Hence, thanks to \eqref{eq:general_cov_double}, we obtain the following measure of double (prior-likelihood) local sensitivity
\begin{align}
\label{eq:delta_double}
   \delta^{h_\alpha}_n= \cov_{\Pi^n}\Big[g(\vartheta), \log F_\vartheta(\vartheta)+\sum_{i=1}^n \log F_X(x_i |  \vartheta)\Big].
\end{align}
For simplicity of exposition, we continue our theoretical development for the sensitivity measures \eqref{eq:power_distortion_lik} and \eqref{eq:delta_lik_normalized}, i.e.\ based on a single distortion of the likelihood. The reasons are both notation simplicity and the fact that the main focus of the paper is to provide measures of local robustness w.r.t.\ the likelihood as to our knowledge there are no well-established proposals in the literature. However, the entire theoretical development applies identically to both \eqref{eq:delta_prior} and \eqref{eq:delta_double}.
Now we obtain another version of the local sensitivity based on $\Tilde{h}_\alpha$.
\begin{proposition}
\label{prop:survival_power_distorsion}
    Let $\Tilde h_\alpha(y)=1-(1-y)^\alpha$ with $\alpha>1$, under the assumptions of Theorem~\ref{thm:lik_dist} it holds
  \begin{align*}
  \delta^{\Tilde h_\alpha}_n  = \cov_{\Pi^n}\Big[g(\vartheta) ,\sum_{i=1}^n \log S_X(x_i |  \vartheta)\Big].
    \end{align*}
\end{proposition}
Under the same assumptions, Corollary~\ref{cor:inequality} yields the following bound 
\begin{align}
\label{eq:bound-CS-surv}
   |\delta^{\Tilde h_\alpha}_n  |  \leq \var_{\Pi^n}[g(\vartheta)]^{1/2} \var_{\Pi^n} \Big[\sum_{i=1}^n\log S_X(x_i |  \vartheta)\Big]^{1/2},
\end{align}
from which we may obtain the normalized version
\begin{align*}
    \overline{\delta}^{\Tilde h_\alpha}_n:= \delta^{h_\alpha}_n\var_{\Pi^n}[g(\vartheta)]^{-1/2}  \var_{\Pi^n} \Big[\sum_{i=1}^n\log S_X(x_i |  \vartheta)\Big]^{-1/2},
\end{align*}
that ranges in $[-1,1]$.
\begin{example}
    \label{ex:prop}
   Let us consider the gamma-exponential setting of Example~\ref{ex:exponential-gamma1} with $g(\vartheta)=\vartheta$. We know that for $\Tilde{h}_\alpha(y)= 1-(1-y)^\alpha$ composing $F_X$ it follows 
   $$
   \E_{\Pi_{\Tilde h_\alpha}^n}[\vartheta]= \frac{a+n}{b+\alpha \sum_{i=1}^n x_i}.
   $$ 
   Therefore, we have
\begin{equation}
\label{eq:analytical-verification}
        \frac{\partial}{\partial \alpha}\E_{\Pi_{\Tilde h_\alpha}^n}[\vartheta]\Big|_{\alpha=1} = -\frac{(a+n)\sum_{i=1}^n x_i}{(b+ \sum_{i}^n x_i)^2} = -\var_{\Pi^n}[\vartheta]\sum_{i=1}^nx_i.
\end{equation}
Now we verify that Proposition~\ref{prop:survival_power_distorsion} gives the same results, thus, we first check the required conditions. First, note that Assumption~\ref{assump:distortion} is satisfied provided $S_X(x_i|\theta)>0$ $\Pi$-almost surely for all $i=1,\dots,n$. Moreover, let $ \delta_0>0$, by triangle inequality, for all $\delta$ such that $|\delta|\leq \delta_0$ and for all $\theta >0$ we have
    \begin{align*}
         |  \delta^{-1} g(\theta)(\mathcal{L}_{\Tilde h_{1+\delta}}(\theta)-\mathcal{L}_{\Tilde h_1}(\theta))|&
     \leq |\delta^{-1}\theta^{n+1}(1+\delta)^n|+ |\delta^{-1}\theta^{n+1}|,
    \end{align*}
    which implies that the conditions of Theorem~\ref{thm:lik_dist} are satisfied since the Gamma distribution has finite moments of any order.
 Consequently, Proposition~\ref{prop:survival_power_distorsion} applies and we obtain
    \begin{align*}
       \delta^{\Tilde h_\alpha}_n&=\E_{\Pi^n}\Big[\vartheta \sum_{i=1}^n-x_i\vartheta\Big]-\E_{\Pi^n}[\vartheta]\E_{\Pi^n}\Big[-\sum_{i=1}^n x_i\vartheta\Big]=-\var_{\Pi^n}[\vartheta]\sum_{i=1}^nx_i,
    \end{align*}
   matching the results of display \eqref{eq:analytical-verification}. Moreover, it is worth noting that
    \begin{align*}
        \var_{\Pi^n}\Big[\sum_{i=1}^n\log S_X(x_i |  \vartheta)\Big]= \var_{\Pi^n}\Big[-\sum_{i=1}^n x_i\vartheta\Big]= \var_{\Pi^n}[\vartheta] \Big(\sum_{i=1}^nx_i\Big)^2.
    \end{align*}
    Hence,
    \begin{align*}
    \var_{\Pi^n}[\vartheta]^{1/2}\var_{\Pi^n}\Big[\sum_{i=1}^n\log S_X(x_i |  \vartheta)\Big]^{1/2} = \var_{\Pi^n}[\vartheta] \sum_{i=1}^nx_i,
    \end{align*}
and, as a consequence, the bound in \eqref{eq:bound-CS-surv} is sharp.
    \end{example}
An interesting extension to the posterior predictive local sensitivity defined in \eqref{eq:predictive_local_sens} can be easily obtained under power distortion; we state this result only for likelihood distortion through $h_\alpha$ defined in \eqref{eq:pow_dist}.
\begin{proposition}
\label{prop:predictve_local_sensitivity}
    Let $h_\alpha(y)= y^\alpha$, with $\alpha>1$, be a distortion function composing $F_X$. Let $\delta_0>0$ and assume that, for $g$ measurable and $G\in L^1(\lambda \otimes \Pi)$,
   $$
          |\delta^{-1}g(x)( f_{X_{h_{1+\delta}}}(x|\theta)\mathcal{L}_{h_{1+\delta}}(\theta) - f_{X_{h_1}}(x|\theta)\mathcal{L}_{h_1}(\theta))|\leq G(x,\theta), \quad |\delta|\leq \delta_0
   $$
   holds $\lambda \otimes \Pi$-almost surely, then
\begin{equation}
\begin{aligned}
\label{eq:predictive_local_sens_power}
       \delta_{n,\mathrm{pred}}^{h_\alpha}=\int_{\mathbb{R}^d} g(x) \E_{\Pi^n}\Big[ f_X(x|\vartheta)\Big(1+n+\log F_X(x|\vartheta)+ \sum_{i=1}^n \log F_X(x_i|\vartheta)-R_n \Big)\Big] d\lambda(x)
\end{aligned}
\end{equation}
    where
    \begin{align*}
    R_n= \frac{\frac{\partial}{\partial \alpha}  f_{X_{h_\alpha,1:n}}(x_{1:n})\Big|_{\alpha=1}}{ f_{X_{1:n}}(x_{1:n})}= 
    \frac{\int_\Theta\mathcal{L}(\theta) \sum_{i=1}^n \log F_X(x_i |  \theta)d\Pi(\theta)}{\int_\Theta\mathcal{L}(\theta) d\Pi(\theta)}.
\end{align*}
\end{proposition}
It is not difficult to note that $R_n$ is an evidence ratio between the slope of the distorted evidence (marginal distorted likelihood) at $\alpha=1$ and the original evidence.
It is immediately apparent that \eqref{eq:predictive_local_sens_power} is a rather more complicated expression that captures perturbations on the predictive behaviour of the model and, therefore, it should be regarded as a measure of predictive robustness; however, its numerical approximation is not straightforward and is left for future research.

\subsection{Local sensitivity under censoring distortion}
In this subsection, we briefly consider the censoring distortion functions.
\begin{definition}
      Let $\alpha\in (0,1) $, a censoring distortion is induced by the functions $h_\alpha:[0,1]\to [0,1]$ and $\Tilde{h}_\alpha:[0,1]\to [0,1]$ given by
      \begin{align*}
        y\mapsto{h}_\alpha(y)=\frac{y-\alpha}{1-\alpha}\vee 0, \quad \quad y\mapsto\Tilde{h}_\alpha(y)=\frac{y}{\alpha} \wedge 1.
    \end{align*}
\end{definition}
The censoring distortion induces a truncation of the original random variable of the type $X_{ h_\alpha}=X\mathbf{1}{\{X> F_X^{-1}(\alpha)\}}$ and $X_{\Tilde h_\alpha}=X\mathbf{1}{\{X\leq F^{-1}_X(\alpha)\}}$. Therefore, in the local sensitivity framework, it is a peculiar way to measure robustness to arbitrary small truncation of the original model. Following the same path of power distortion, we obtain the following propositions.

\begin{proposition}
\label{prop:censoring_distortion_1}
Let ${h}_\alpha(y)=\frac{y-\alpha}{1-\alpha}\vee 0 $ with $\alpha\in (0,1)$, under the assumptions of Theorem~\ref{thm:lik_dist}, it holds
$ \delta^{ h_\alpha}_n =0$.
\end{proposition}

\begin{proposition}
    \label{prop:censoring_distortion_2}
Let $\Tilde h_\alpha(y)=\frac{y}{\alpha} \wedge 1 $ with $\alpha\in (0,1)$, under the assumptions of Theorem~\ref{thm:lik_dist}, it holds
$\delta^{\Tilde h_\alpha}_n =0$.
\end{proposition}
It is at once apparent that censoring distortion does not provide a useful local robustness analysis as the involved distortion functions induce a non-smooth shock in the original model. Similar results can be shown for prior distortion and prior-likelihood double distortion.

\subsection{Local sensitivity under skewing distortion}
A further distortion that we consider is related to the family of skewed distributions; see \cite{azzalini1985class} and \cite{ferreira2006constructive}. In particular, provided $X$ is symmetric w.r.t.\ 0, we are interested in introducing a certain amount of skewness to its distribution, leading to the family of skewed distributions with density
\begin{align*}
    f_{X,\alpha}(x|\theta)= 2 f_{X}(x|\theta)F_X(\alpha x|\theta).
\end{align*}
Note that $f_{X,\alpha}(x|\theta)$ is left skewed for $\alpha <0 $, and right skewed for $\alpha>0$.
In \cite{distorsion_priors2016} it is shown that the skewed family is a proper neighbourhood of the original distribution in terms
of the likelihood ratio order.
For our purposes, skewness might be induced via the following distortion function.
\begin{definition}
Let $\alpha \in \mathbb{R}\setminus \{0\}$, the skewing distortion of $X$ is induced by the function $h_\alpha: [0,1]\to [0,1]$ given by
\begin{align}
\label{eq:skewing_distortion}
   y\mapsto h_\alpha(y)= \int_{-\infty}^{F^{-1}_X(y)}2f_{X}(x|\theta)F_X(\alpha x|\theta)d\lambda(x). 
\end{align}
\end{definition}
In the following proposition, we obtain an expression for the local sensitivity under skewing distortion. 
\begin{proposition}
    \label{prop:skewing_distortion}
    Let $h_\alpha$ be the skewing distortion function defined in \eqref{eq:skewing_distortion}, under the assumptions of Theorem~\ref{thm:lik_dist}, it holds
    \begin{align}
\label{eq:skewing_lik}
    \delta^{h_\alpha}_n = \cov_{\Pi^n}\Big[g(\vartheta) ,2 f_X(0|\vartheta)\sum_{i=1}^n x_i\Big].
\end{align}
\end{proposition}
An interpretation for \eqref{eq:skewing_lik} is that it measures the linear relationship between the (transformed) parameter vector $g(\vartheta)$ and $2 f_X(0|\vartheta)\sum_{i=1}^n x_i$, where the latter explicitly depends on the density function of $X$ evaluated at 0. Thus, whenever the model exhibits a strong linear relationship between the parameter and the density evaluated at 0 we may conclude a lack of robustness to modifications of the distribution skewness.
Extensions to prior and prior-likelihood double distortion can be trivially obtained, and owing to Corollary~\ref{cor:inequality}, a normalized version can be obtained as in \eqref{eq:delta_lik_normalized}.

\subsection{Approximation of local sensitivity measures via posterior sampling}
We now briefly discuss how the proposed local sensitivity measure can be easily approximated once a sample from the posterior distribution is available. Without loss of generality, we discuss this only for likelihood distortion. Extensions to prior distortion and prior-likelihood double distortion are straightforward.

Given a sample $X_{1:n}=x_{1:n}$ which at the moment is considered as fixed, $\delta^{h_\alpha}_n$ might be evaluated once a sample $\vartheta^{(1)},\dots, \vartheta^{(M)}$ from the posterior distribution is available, e.g. via standard Monte Carlo or Markov Chain Monte Carlo (MCMC) sampling. Indeed, for $M$ large enough, thanks to the Strong Law of Large Numbers or the Ergodic Theorem,  $\delta^{h_\alpha}_n$ can be approximated as
\begin{align}
\label{eq:approx}
    \delta^{h_\alpha}_n\approx{\frac{1}{M} \sum_{m=1}^M g(\vartheta^{(m)})T(\vartheta^{(m)}) - \frac{1}{M} \sum_{m=1}^M g(\vartheta^{(m)}) \frac{1}{M} \sum_{m=1}^M T(\vartheta^{(m)})  }
\end{align}
where $$T(\vartheta^{(m)}):=\sum_{i=1}^n \frac{\partial}{\partial \alpha}  \log  h'_\alpha (F_X(x_i |  \vartheta^{(m}))\Big|_{\alpha=\alpha_0} .$$
For instance, under power distortion, thanks to \eqref{eq:power_distortion_lik} it is at once apparent that $\delta^{h_\alpha}_n$ can be approximated replacing $T(\vartheta^{(m)})= \sum_{i=1}^n \log F_X(x_i |  \vartheta^{(m)})$. Equivalently $\delta^{\Tilde h_\alpha}_n$ can be approximated replacing $T(\vartheta^{(m)})= \sum_{i=1}^n \log S_X(x_i | \vartheta^{(m)})$. 
Moreover, under a skewing distortion, due to \eqref{eq:skewing_lik}, it suffices to set $T(\theta^{(m})=2 f_{X}(0|\vartheta^{(m})\sum_{i=1}^nx_i$.
It is also worth remarking that due to the theoretical results obtained, we do not need to sample the distorted posterior with acceptance-rejection algorithms as in \cite{distorsion_priors2016}.

\section{Asymptotic results}\label{sec:asymptotic_behaviour}
In this section, we study the asymptotic behaviour of $\delta^{h_\alpha}_n$ as $n\to \infty$.
Our results are derived under the assumption that the posterior distribution is consistent at some $\theta_0\in \Theta$, which means it converges towards a point mass at some rate of contraction. More precisely, given a suitable semi-metric $d$ on $\Theta$, a decreasing sequence $\varepsilon_n$ is a contraction rate at $\theta_0$ w.r.t.\ $d$ for the posterior distribution $\Pi^n$ if, for every $M_n\to \infty$, it holds that $\Pi^n(\theta: d(\theta,\theta_0) \geq M_n \varepsilon_n)=\Pi(\theta: d(\theta,\theta_0)\geq  M_n \varepsilon_n |  X_{1:n})\to 0$ in probability (or almost surely) under the product probability measure $P^{n}_{\theta_0}$ as $n\to \infty$. As usual, the requirement ``for every $M_n\to \infty$'' should be understood as ``for any arbitrarily slow $M_n\to \infty$''; see \cite{ghosal2017fundamentals} as a reference for the general theory on this topic. Note also that in the following we relax the assumption of non-randomness of the sample, thus, for instance, $g(\vartheta)$ and $F_X(X_i | \vartheta)$ are seen as random variables with randomness given by both $X_i$ and $\vartheta$.
The asymptotic results are derived under the following assumptions.
\begin{assumption}
    \label{assump:ftheta}
    The functions $g:\Theta \to \mathbb{R}^s$ (in Definition~\ref{def:local_sensitivity}) and $r_i:\Theta \to \mathbb{R}$ given by
    \begin{align*}
         \vartheta \mapsto r_i(\vartheta):= \frac{\partial}{\partial \alpha}\log  (h'_\alpha(F_{X}(X_i|\vartheta)))\Big|_{\alpha=\alpha_0}, \ \ i=1,\dots,n,
    \end{align*}
are differentiable with bounded partial derivatives at $\theta_0$.
\end{assumption}

\begin{assumption}
\label{assump:metric}
    The posterior distribution $\Pi^n$ is consistent at $\theta_0\in \Theta$ with contraction rate $\varepsilon_n$ w.r.t.\ a bounded semi-norm-induced semi-metric $d$ such that the map $\theta \mapsto d^q(\theta,\theta_0)$ is convex for some $q\geq 1$, and $\Pi^n(\theta: d(\theta,\theta_0)> M_n \varepsilon_n)= O{\scriptscriptstyle P_{\theta_0}^n}(M_n\varepsilon_n)$ for every $M_n\to \infty$.
\end{assumption}

The following result concerns the asymptotic behaviour of $\delta_n^{h_\alpha}$ and shows that it depends on the contraction rate of the posterior distribution.
\begin{theorem}
\label{thm:asymp}
     Suppose that Assumptions~\ref{assump:ftheta}~and~\ref{assump:metric} hold, then under the assumptions of Theorem~\ref{thm:lik_dist} 
     $$  \delta^{h_\alpha}_n=  \cov_{\Pi^n} \Big[g(\vartheta), \sum_{i=1}^n \frac{\partial}{\partial \alpha}  \log  h'_\alpha (F_X(X_i |  \vartheta))\Big|_{\alpha=\alpha_0}\Big]= O{\scriptscriptstyle P^{n}_{\theta_0}}(n\varepsilon_n^2)
     $$
     as $n\to \infty$.
\end{theorem}
Theorem~\ref{thm:asymp} highlights that the asymptotic behaviour of $\delta^{h_\alpha}_n$ is related to the squared contraction rate of the posterior distribution. A direct consequence is that the sequence $\{\delta_n^{h_\alpha}\}_{n\in \mathbb{N}}$ is tight under parametric contraction rate, i.e.\ for $\varepsilon_n=n^{-1/2}$ we get $\delta_n^{h_\alpha}=O{\scriptscriptstyle P^{n}_{\theta_0}}(1)$. 
When the contraction rate is slower than $ n^{-1/2}$, it yields that $\{\delta_n^{h_\alpha}\}_{n\in \mathbb{N}}$ is not guaranteed to be tight.
Note also that the asymptotic results given above can be straightforwardly extended to non-i.i.d.\ settings with a slight effort of notation.

To conclude the asymptotic results, we give a weak convergence result that holds for models whose posterior distribution achieves parametric contraction rate.

\begin{theorem}
\label{thm:clt}
    Let $p:=\sup\{j>0: \E[|X|^j]<\infty\}\geq 2 $ and suppose that the function
    $$x \mapsto \psi(x):=  \nabla_{\scriptscriptstyle \theta} \Big(\frac{\partial}{\partial \alpha}\log  (h'_\alpha(F_{X}(x|\theta)))\Big|_{\alpha=\alpha_0}\Big)\Big|_{\theta=\theta_0}$$
    is $O(x^{p/2})$.
    Then, under the assumptions of Theorem~\ref{thm:asymp}, there exist $\delta_0 \in \mathbb{R}^s$ and $\Sigma \in \mathbb{R}^{s\times s}$ such that as $n\to \infty$ 
    \begin{align} \label{clt}
        Z_n = \sqrt{n}\Sigma^{-1/2}(\delta_n ^{h_\alpha}- \delta_0)
    \end{align}
    converges in distribution to a standard multivariate Gaussian r.v. $Z\sim \mathcal{N}(0,I)$.
\end{theorem}

\begin{example}
    \label{ex:asympt}
Continuing Example~\ref{ex:prop} we proceed to apply the asymptotic results obtained in Theorems~\ref{thm:asymp}~and~\ref{thm:clt}. It is easy to check that Assumption~\ref{assump:ftheta} is satisfied since
    \begin{align*}
        \frac{\partial}{\partial \alpha}\log  (\Tilde h_\alpha'(F_{X}(X_1|\theta)))\Big|_{\alpha=1}&= \frac{\partial}{\partial \alpha} \Big(\log\alpha + (\alpha-1)\log(1-F_{X}(X_1|\theta))\Big) \Big|_{\alpha=1} = 1-X_1\theta.
    \end{align*}
Moreover, it is immediate to see that by Lemma 8.2 of \cite{ghosal2017fundamentals} the posterior is consistent at the true $\theta_0\in \Theta$ with contraction rate $\varepsilon_n=n^{-1/2}$ in the $L^2$-metric, and the condition $\Pi^n(\theta: \|\theta-\theta_0\|_2> M_n n^{-1/2})= O{\scriptscriptstyle P_{\theta_0}^n}(M_nn^{-1/2})$ is satisfied by the classical Bernstein-von-Mises Theorem~(\citealp{lecamasymptotic}) and the Chernoff exponential tail bound for the chi-squared distribution.
Therefore, Theorem~\ref{thm:asymp} applies and it yields that $ \delta^{\Tilde h_\alpha}_n=O{\scriptscriptstyle P^{n}_{\theta_0}}(1)$. This is easily confirmed by observing that as $n \to \infty$
\begin{equation}
\label{eq:analytical-verification1}
  \delta^{\Tilde h_\alpha}_n =  -\var_{\Pi^n}[\vartheta]\sum_{i=1}^nX_i\sim -\frac{n}{\sum_{i=1}^nX_i}\to \overset{P^n_{\theta_0}}{\to} -\theta_0.
\end{equation}
Finally, to apply Theorem~\ref{thm:clt} we check that 
$$\psi(x)=\nabla_{\scriptscriptstyle \theta} \Big(\frac{\partial}{\partial \alpha}\log  (\Tilde h'_\alpha(F_{X}(x|\theta)))\Big|_{\alpha=\alpha_0}\Big)\Big|_{\theta=\theta_0}=-x=O(x^{p/2})$$
for $p=\sup\{j>0: \E[|X|^j]<\infty\}$, because for the exponential distribution $p=\infty$.
Thus, let $\delta_0=-\theta_0$ and $\Sigma=\theta_0^2$, by Theorem~\ref{thm:clt} as $n\to\infty$
    \begin{align*}
           Z_n=\sqrt{n} \Sigma^{-1/2} (\delta^{\Tilde h_\alpha}_n-\delta_0)\sim\frac{\sqrt{n}}{\theta_0} \Big(\theta_0-\frac{n}{\sum_{i=1}^n X_i}\Big)
    \end{align*}
    converges in distribution to a standard Gaussian r.v. $Z\sim \mathcal{N}(0,1)$.
\end{example}

\section{Numerical experiments}
\label{sec:num}
In this section, we provide some numerical experiments, both on simulated and real data, to illustrate the theory and to show its applicability to Bayesian statistical modelling.
\subsection{Simulation results}
As a first simulation experiment, we validate the theory that we derived for the basic gamma-exponential model presented in Example~\ref{ex:asympt}.
We draw i.i.d.\ samples with varying sizes $n\in \{ j ^3: j = 3,\dots, 10 \}$ from an exponential distribution with parameter $\theta_0=0.5$
and compute the local sensitivity ${\delta}^{\Tilde h_\alpha}_n$ with its 95\% confidence interval under power distortion. For the computation, we used the approximation given in \eqref{eq:approx} with $M=10000$ samples from the posterior distribution. Figure~\ref{fig:gammaexp} shows that the local sensitivity measure converges to $-\theta_0$, as illustrated analytically in Example~\ref{ex:asympt}.
\begin{figure}[!ht]
    \centering
    \includegraphics[]{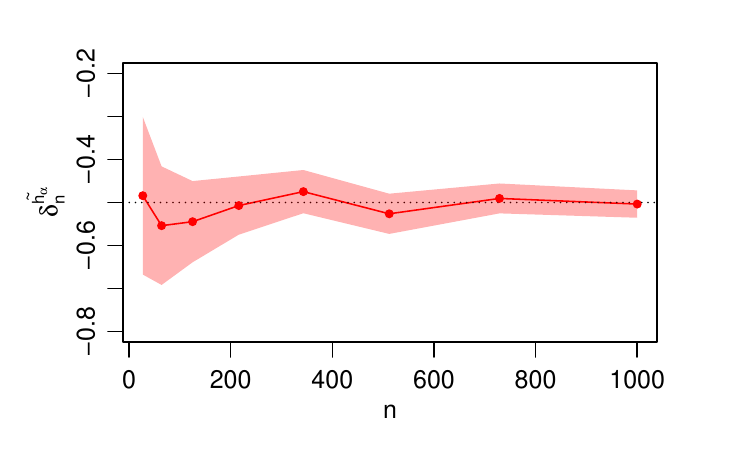}
    \caption{Approximated ${\delta}^{\Tilde h_\alpha}_n$ under power distortion of the likelihood with varying sample sizes. The dotted line corresponds to the theoretical value $-\theta_0=-0.5$ and the shaded region is the 95\% confidence interval obtained by an application of Theorem~\ref{thm:clt}.}
    \label{fig:gammaexp}
\end{figure}

The second simulation experiment that we present is an application to model selection. We generate random samples of size $n\in \{50,200\}$ with the following data-generating processes (DGP): $\mathrm{Gamma}(1,1)$, $\mathrm{LogNormal}(0,1)$ and $\mathrm{Exponential}(1)$. Successively, we fit those distributions on the obtained samples and compute ${\delta}^{ h_\alpha}_n$ under power distortion using the approximation in \eqref{eq:approx} with $M=2000$ samples from the posterior distribution obtained through standard Monte Carlo or MCMC sampling. We consider the case where $g(\vartheta)= \vartheta_j$ with $j=1,\dots,k$. The results are reported in Table~\ref{tab:n1} ($n=50$) and Table~\ref{tab:n2} ($n=200$). In several cases, we notice that ${\delta}^{ h_\alpha}_n$ is larger in absolute value when we fit a distribution that is different from the DGP. In particular, in the case of a multi-parameter distribution, if we consider the average absolute ${\delta}^{ h_\alpha}_n$, in several cases we conclude that higher robustness is shown when the distribution that coincides with the true DGP is fitted.
Moreover, for both $n=50$ and $n=200$ the Gamma distribution is always more robust when the true DGP is $\mathrm{Exponential}(1)$. We could also note that if we focus on the whole parameter vector of the two multi-parameter distributions and compare its associated average absolute local sensitivity to that of the rate parameter of the Exponential distribution, we observe that the latter possesses good robustness properties, even though its goodness-of-fit may be smaller. One possible reason is that it is less sensitive to outliers, as it does not have the additional shape parameter. On the other hand, if we compare the Gamma and the Exponential distributions but restrict our focus to the rate parameters, we can conclude that the Gamma distribution is more robust.
\begin{table}[!ht]
  \centering
  \begin{tabular}{lccc}
    \hline
    DGP $\backslash$ Distribution & $\mathrm{Gamma}(\alpha,\beta)$  &$\mathrm{LogNormal}(\alpha,\beta)$  & $\mathrm{Exponential}(\theta)$  \\
    \hline
$\mathrm{Gamma}(1,1)$       &(-0.99, -0.25) & (-1.17, 0.47) & 0.56\\
$\mathrm{LogNormal}(0,1)$   &(-1.19, -0.30) & (-0.80, 0.26) & 0.56\\
$\mathrm{Exponential}(1)$       &(-0.79, -0.22) & (-1.42, 0.52) & 0.62\\
    \hline
  \end{tabular}
    \caption{Approximated ${\delta}^{ h_\alpha}_n$ under power distortion of the likelihood with $n=50$.}
    \label{tab:n1}
\end{table}

\begin{table}[!ht]
  \centering
  \begin{tabular}{lccc}
    \hline
    DGP $\backslash$ Distribution & $\mathrm{Gamma}(\alpha,\beta)$  &$\mathrm{LogNormal}(\alpha,\beta)$  & $\mathrm{Exponential}(\theta)$  \\
    \hline
$\mathrm{Gamma}(1,1)$         & (-0.87, -0.21) & (-1.15, 0.50)& 0.61\\
$\mathrm{LogNormal}(0,1)$     & (-1.13, -0.18) & (-0.85, 0.29)& 0.39\\
$\mathrm{Exponential}(1)$         & (-0.99, -0.28) & (-1.07, 0.41)& 0.66\\
    \hline
  \end{tabular}
    \caption{Approximated ${\delta}^{ h_\alpha}_n$ under power distortion of the likelihood with $n=200$.}
     \label{tab:n2}
\end{table}
\subsection{Applications to real data}
To further illustrate the methodology, we consider applications to reliability and earthquake data. In particular, under the same modelling framework of the second simulation experiment, we fit some parametric models via standard posterior sampling and we measure the local sensitivity w.r.t.\ the likelihood function under power distortion. For all the fitted distributions and datasets, we take the posterior mean as a Bayesian point estimator.

The first dataset that we consider is the Windshield Failure Data reported in \cite{murthy2004weibull}, Table~16.11, which is composed of failure times of 84 aircraft windshields collected with a unit of measurement of 1000h. The data set has previously been analysed in \cite{el2015exponential} and \cite{ijaz2019lomax}. In Table~\ref{tab:data} we report the results of the approximated ${\delta}^{ h_\alpha}_n$ under power distortion of the likelihood. The LogNormal distribution is less sensitive to small perturbations of the likelihood through power distortion if compared to the Gamma, which exhibits a strong sensitivity of the shape parameter. Moreover, the Exponential distribution appears to be more robust; however, as we can see in Figure~\ref{fig:data}, this could be a consequence of a strong lack of fit, which might compromise the significance and effectiveness of robustness diagnostics.

The second dataset that we consider is based on worldwide Earthquake data from January 1, 2000, to December 31, 2023, with a minimum magnitude of 7. The data are publicly available for download on the earthquake section of the United States Geological Survey (USGS) official website. In particular, we consider the distribution of the inter-arrival times between 361 earthquakes measured in months. In this case, from Figure~\ref{fig:data} we can appreciate that, as expected, all three models exhibit a good fit to the data; however, from Table~\ref{tab:data} it is clear that the Gamma and the Exponential distributions are much more robust to small perturbations of the likelihood through power distortion.

\begin{figure}[!ht]
    \centering
    \includegraphics[]{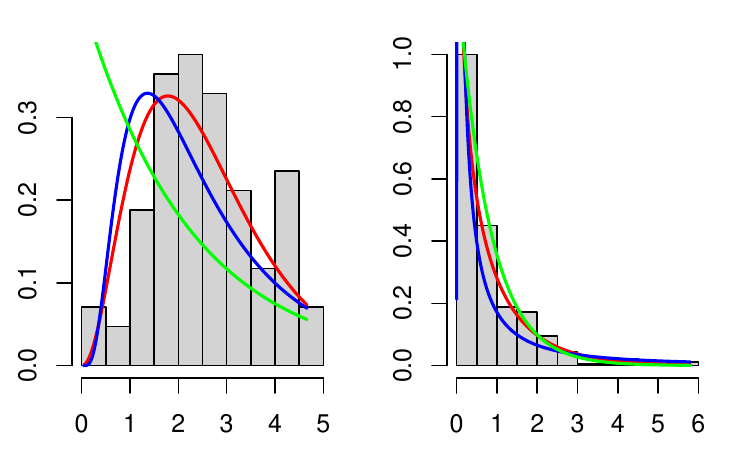}
    \caption{Histograms and fitted densities for the two considered datasets: Windshield Failure Data on the left, Earthquake Data on the right. Colours:  Gamma in red, LogNormal in blue and Exponential in green.}
    \label{fig:data}
\end{figure}
\begin{table}[!ht]
  \centering
  \begin{tabular}{lccc}
    \hline
     Dataset $\backslash$ Distribution & $\mathrm{Gamma}(\alpha,\beta)$  &$\mathrm{LogNormal}(\alpha,\beta)$  & $\mathrm{Exponential}(\theta)$  \\
    \hline
Windshield Failure Data   &  (-3.80, -0.97)  & (-0.58, 0.28) & 0.25\\ 
Earthquake Data    &(-0.60, -0.21) & (-1.66,  0.70) & 0.83\\
    \hline
  \end{tabular}
    \caption{Approximated ${\delta}^{ h_\alpha}_n$ under power distortion of the likelihood for the Windshield Failure Data and the Earthquake Data.}
     \label{tab:data}
\end{table}

\section{Discussion}
\label{sec:conc}
In this paper, we introduced a mathematically tractable approach to derive local sensitivity measures for Bayesian robustness w.r.t.\ the likelihood function. Further extensions to prior distortion and prior-likelihood double distortion are also presented. This new approach, based on the distortion technique, is shown to be convenient in terms of interpretability and computational effort, and it leads to local sensitivity measures that possess suitable asymptotic properties. Moreover, it provides a general framework for an interpretable elicitation of the robustness property that one wants to check.
Numerical experiments showed promising results regarding the applicability of such measures to model selection. 

Now we list some open questions and further developments that we find promising. An interesting future development is the extension of this approach to the global robustness framework and the exploitation of the predictive robustness framework that we introduced in the paper.
A further development that has not received much attention in the literature is to consider the second-order derivative in the definition of the local sensitivity, to capture the curvature of the posterior. A similar concept of local curvature has been introduced in \cite{dey1994robust} for divergence measures based on $\epsilon$-contamination classes and mixture classes.
Finally, following \cite{zhu2001case}, the local sensitivity measures that we introduced could be used as a case deletion diagnostic measure to detect influential observations.

\appendix
\section{Proofs}
\label{sec:proofs}
\subsection*{Proofs for Section~\ref{sec:dist}}
\begin{proof}[Proof of Theorem~\ref{thm:lik_dist}]
    Notice that
    \begin{align*}
        \frac{\partial}{\partial \alpha }\prod_{i=1}^n h'_\alpha  (F_X(x_i |  \theta))&= \frac{\partial }{\partial \alpha} \exp\Big( \sum_{i=1}^n \log  h'_\alpha(F_X(x_i |  \theta))\Big)\\
        &= \prod_{i=1}^n  h'_\alpha (F_X(x_i |  \theta)) \frac{\partial}{\partial \alpha}\Big(\sum_{i=1}^n \log  h'_\alpha(F_X(x_i |  \theta))  \Big),
    \end{align*}
and by definition of $\alpha_0$ together with Assumption~\ref{assump:distortion} we have that
\begin{align*}
    \lim_{\alpha \to \alpha_0} \prod_{i=1}^n  h'_\alpha (F_X(x_i |  \theta))  = 1,
\end{align*}
which clearly holds also when $X$ is discrete by Remark~\ref{rmk:discrete}.  
Hence, the Dominated Convergence Theorem yields that 
 \begin{align*}
         \frac{\partial}{\partial \alpha}\E_{\Pi_{h_\alpha}^n}[g(\vartheta)]\Big|_{\alpha=\alpha_0}&= \frac{\partial}{\partial \alpha} \frac{\int_\Theta g(\theta) \mathcal{L}_{h_\alpha}(\theta) d\Pi(\theta)  }{\int_\Theta \mathcal{L}_{h_\alpha}(\theta) d\Pi(\theta) }\Big|_{\alpha=\alpha_0}\\
    &=\frac{\int_\Theta  g(\theta) \frac{\partial}{\partial \alpha} \Big(\sum_{i=1}^n \log  h'_\alpha (F_X(x_i |  \theta)) \Big)\Big|_{\alpha=\alpha_0} d \Pi(\theta |  x_{1:n}) }{\int_\Theta d \Pi{(\theta |  x_{1:n})} } \\
    &- \frac{\int_\Theta g(\theta)   d \Pi(\theta |  x_{1:n}) }{\int_\Theta d \Pi{(\theta |  x_{1:n})} }\frac{\int_\Theta \frac{\partial}{\partial \alpha} \Big(\sum_{i=1}^n \log  h'_\alpha (F_X(x_i |  \theta)) \Big)\Big|_{\alpha=\alpha_0} d \Pi(\theta |  x_{1:n}) }{\int_\Theta d \Pi{(\theta |  x_{1:n})} } \\
    &= \cov_{\Pi^n}\Big[g(\vartheta), \sum_{i=1}^n\frac{\partial}{\partial \alpha} \log  h'_\alpha (F_X(x_i |  \vartheta)) \Big|_{\alpha=\alpha_0}\Big],
    \end{align*}
 where we used that $ d \Pi_{h_\alpha}{(\theta |  x_{1:n})} = d \Pi{(\theta |  x_{1:n})}$ for $\alpha = \alpha_0$. The proof is concluded.
\end{proof}

\begin{proof}[Proof of Theorem~\ref{thm:prior_distortion}]
    Using the same arguments as the proof of Theorem~\ref{thm:lik_dist} we apply the Dominated Convergence Theorem and obtain 
     \begin{align*}
         \frac{\partial}{\partial \alpha}\E_{\Pi_{h_\alpha}^n}[g(\vartheta)]\Big|_{\alpha=\alpha_0}&= \frac{\partial}{\partial \alpha} \frac{\int_\Theta g(\theta) \mathcal{L}(\theta) d\Pi_{h_\alpha}(\theta)  }{\int_\Theta \mathcal{L}(\theta) d\Pi_{h_\alpha}(\theta) }\Big|_{\alpha=\alpha_0}\\
    &=\frac{\int_\Theta  g(\theta) \frac{\partial}{\partial \alpha}  \log  h'_\alpha (F_\vartheta(\theta)) \Big|_{\alpha=\alpha_0} d \Pi(\theta |  x_{1:n}) }{\int_\Theta d \Pi{(\theta |  x_{1:n})} } \\
    &- \frac{\int_\Theta g(\theta)   d \Pi(\theta |  x_{1:n}) }{\int_\Theta d \Pi{(\theta |  x_{1:n})} }\frac{\int_\Theta \frac{\partial}{\partial \alpha}  \log  h'_\alpha (F_\vartheta(\theta)) \Big|_{\alpha=\alpha_0} d \Pi(\theta |  x_{1:n}) }{\int_\Theta d \Pi{(\theta |  x_{1:n})} } \\
    &= \cov_{\Pi^n}\Big[g(\vartheta), \frac{\partial}{\partial \alpha}  \log  h'_\alpha (F_\vartheta(\theta)) \Big|_{\alpha=\alpha_0}\Big].
    \end{align*}
  where we used that $ d \Pi_{h_\alpha}{(\theta |  x_{1:n})} = d \Pi{(\theta |  x_{1:n})}$ for $\alpha = \alpha_0$. The proof is concluded.
\end{proof}

\begin{proof}[Proof of Theorem~\ref{thm:double_distortion}]
    The proof is a combination of the arguments given in the proofs of Theorems~\ref{thm:lik_dist}~and~\ref{thm:prior_distortion}.
     Thus, we employ the Dominated Convergence Theorem to obtain
      \begin{align*}
         &\frac{\partial}{\partial \alpha}\E_{\Pi_{h_\alpha}^n}[g(\vartheta)]\Big|_{\alpha=\alpha_0}= \frac{\partial}{\partial \alpha} \frac{\int_\Theta g(\theta) \mathcal{L}_{h_\alpha}(\theta) d\Pi_{h_\alpha}(\theta)  }{\int_\Theta \mathcal{L}_{h_\alpha}(\theta) d\Pi_{h_\alpha}(\theta) }\Big|_{\alpha=\alpha_0}\\
    &=\frac{\int_\Theta  g(\theta) \frac{\partial}{\partial \alpha} \Big(\log  h'_\alpha (F_\vartheta(\vartheta))+ \sum_{i=1}^n \log  h'_\alpha (F_X(x_i |  \vartheta)) \Big)\Big|_{\alpha=\alpha_0} d \Pi(\theta |  x_{1:n}) }{\int_\Theta d \Pi{(\theta |  x_{1:n})} } \\
    &- \frac{\int_\Theta g(\theta)   d \Pi(\theta |  x_{1:n}) }{\int_\Theta d \Pi{(\theta |  x_{1:n})} }\frac{\int_\Theta \frac{\partial}{\partial \alpha} \Big( \log  h'_\alpha (F_\vartheta(\vartheta))+ \sum_{i=1}^n \log  h'_\alpha (F_X(x_i |  \vartheta)) \Big)\Big|_{\alpha=\alpha_0} d \Pi(\theta |  x_{1:n}) }{\int_\Theta d \Pi{(\theta |  x_{1:n})} } \\
    &= \cov_{\Pi^n}\Big[g(\vartheta), \frac{\partial}{\partial \alpha} \Big(\log  h'_\alpha (F_\vartheta(\vartheta))+ \sum_{i=1}^n \log  h'_\alpha (F_X(x_i |  \vartheta)) \Big)\Big|_{\alpha=\alpha_0}\Big],
    \end{align*}
 where we used that $ d \Pi_{h_\alpha}{(\theta |  x_{1:n})} = d \Pi{(\theta |  x_{1:n})}$ for $\alpha = \alpha_0$. The proof is concluded.
\end{proof}

\subsection*{Proofs for Section~\ref{sec:classes}}
\begin{proof}[Proof of Proposition~\ref{prop:power_distorsion}]
The proof is a consequence of Theorem~\ref{thm:lik_dist}, however, we illustrate the intermediate steps to exemplify the derivation under the specified power distortion. 
As a first remark, by an explicit calculation, we observe that
\begin{align*}
    \frac{\partial}{\partial \alpha} \prod_{i=1}^n F_X(x_i |  \theta)^{\alpha-1}&= \frac{\partial}{\partial \alpha} \exp\Big((\alpha-1)\sum_{i=1}^n \log F_X(x_i |  \theta)\Big)\\
    &= \sum_{i=1}^n \log F_X(x_i |  \theta) \prod_{i=1}^n F_X(x_i |  \theta)^{\alpha-1}.
\end{align*}
Assumption~\ref{assump:distortion} implies the requirement $F_X(x_i|\theta)\in (0,1]$ $\Pi$-almost surely for all $i=1,\dots,n$, therefore, it follows
\begin{align*}
      &\frac{\partial}{\partial \alpha}\E_{\Pi_{h_\alpha}^n}[g(\vartheta)]\Big|_{\alpha=1}= \frac{\partial}{\partial \alpha} \frac{\int_\Theta \alpha^n g(\theta )(\prod_{i=1}^n F_X(x_i |  \theta)^{\alpha-1}f_X(x_i |  \theta))d\Pi(\theta)}{\int_\Theta \alpha ^n (\prod_{i=1}^n F_X(x_i |  \theta)^{\alpha-1}f_X(x_i |  \theta)) d\Pi(\theta)}\Big|_{\alpha=1}\\
     &= \frac{\int_\Theta \sum_{i=1}^n \log F_X(x_i |  \theta)g(\theta)  (\prod_{i=1}^n  f_X(x_i |  \theta)) d\Pi (\theta)}{\int_\Theta (\prod_{i=1}^n  f_X(x_i |  \theta))   d\Pi(\theta)} \\
     & - \frac{\int_\Theta g(\theta) (\prod_{i=1}^n f_X(x_i |  \theta))  d\Pi (\theta)}{\int_\Theta (\prod_{i=1}^n  f_X(x_i |  \theta) )  d\Pi(\theta)} \frac{\int_\Theta \sum_{i=1}^n \log F_X(x_i |  \theta) (\prod_{i=1}^n  f_X(x_i |  \theta) ) d\Pi (\theta)}{\int_\Theta (\prod_{i=1}^n f_X(x_i |  \theta)  ) d\Pi(\theta)}\\
&=\int_\Theta g(\theta)\sum_{i=1}^n \log F_X(x_i |  \theta)d\Pi(\theta |  x_{1:n})-\int_\Theta g(\theta)d\Pi(\theta |  x_{1:n})\int_\Theta \sum_{i=1}^n \log F_X(x_i |  \theta)d\Pi(\theta |  x_{1:n})\\
&=\cov_{\Pi^n}\Big[g(\vartheta) ,\sum_{i=1}^n \log F_X(x_i |  \vartheta)\Big],
\end{align*}
where we applied the Dominated Convergence Theorem.
The result aligns with Theorem~\ref{thm:lik_dist}. Indeed, notice that  $$\frac{\partial}{\partial \alpha} \log h'_\alpha (y) = \frac{\partial}{\partial \alpha} \log (\alpha y ^{\alpha-1}) =
\frac{1}{\alpha} + \log y,
$$
is continuous in $\alpha$ and bounded for $y \in (0,1]$, so that clearly $h_\alpha(y) = y^\alpha$ satisfies Assumption~\ref{assump:distortion}.
Since
  \begin{align*}
   \sum_{i=1}^n\frac{\partial}{\partial \alpha}  \log  h'_\alpha (F_X(x_i |  \theta))\Big|_{\alpha=1}&= \sum_{i=1}^n   \frac{\partial}{\partial \alpha} \log (\alpha F_X(x_i|\theta)^{\alpha-1})\Big|_{\alpha=1} \\
   &= \sum_{i=1}^n \Big(\frac{1}{\alpha}+\log F_X(x_i|\theta)\Big)\Big|_{\alpha=1}\\
   &= n +\sum_{i=1}^n \log F_X(x_i|\theta),
    \end{align*}
by Theorem~\ref{thm:lik_dist} it follows that 
$$\delta_n^{h_\alpha} = \cov_{\Pi^n}\Big[g(\vartheta) , n + \sum_{i=1}^n \log F_X(x_i |  \vartheta)\Big] = \cov_{\Pi^n}\Big[g(\vartheta) ,\sum_{i=1}^n \log F_X(x_i |  \vartheta)\Big].  $$
\end{proof}

\begin{proof}[Proof of Proposition~\ref{prop:survival_power_distorsion}]
Note that applying a distortion to $F_X$ using $\Tilde h_\alpha(y)= 1-(1-y)^\alpha$ gives the same results as Proposition~\ref{prop:power_distorsion} in terms of $S_X$. In fact, thanks to \eqref{eq:survival} we have that
\begin{align*}
    \frac{\partial }{\partial x} F_{X_{\Tilde h_\alpha}} (x |  \theta)&=\frac{\partial}{\partial x}( 1- (1-F_X(x |  \theta))^\alpha)= \alpha S_X(x |  \theta)^{\alpha-1}f_X(x |  \theta).
\end{align*}
Now, it suffices to notice
\begin{align*}
    \frac{\partial}{\partial \alpha} \prod_{i=1}^n S_X(x_i |  \theta)^{\alpha-1}&= \frac{\partial}{\partial \alpha} \exp \Big((\alpha-1)\sum_{i=1}^n \log S_X(x_i |  \theta)\Big)\\
    &= \sum_{i=1}^n \log S_X(x_i |  \theta) \prod_{i=1}^n S_X(x_i |  \theta)^{\alpha-1},
\end{align*}
and the proof becomes analogous to the proof of Proposition~\ref{prop:power_distorsion}. To conclude, note that $$
\frac{\partial}{\partial \alpha} \log \Tilde{h}'_\alpha(y) = \frac{\partial}{\partial \alpha} \log (\alpha(1-y)^{\alpha-1}) =\frac{1}{\alpha} + \log (1 - y)
$$ is continuous in $\alpha$ and bounded for $y \in [0,1)$, therefore, $\Tilde h_\alpha(y)= 1-(1-y)^\alpha$ clearly satisfies Assumption~\ref{assump:distortion} and the thesis follows directly from Theorem~\ref{thm:lik_dist}.
\end{proof}

\begin{proof}[Proof of Proposition~\ref{prop:predictve_local_sensitivity}]
    Applying the Dominated Convergence Theorem we obtain
    \begin{align*}
      &\delta_{n,\mathrm{pred}}^{h_\alpha}
      =\frac{\partial}{\partial \alpha} \E_{\Bar{P}^{n}_{h_\alpha}}[g(X)]\Big|_{\alpha=\alpha_0}\\
      &= \frac{\partial}{\partial \alpha} \Big(\int_{\mathbb{R}^d}g(x) d\Bar{P}^{n}_{h_\alpha} (x)\Big)\Big|_{\alpha=\alpha_0} \\
      &= \frac{\partial}{\partial \alpha} \Big(\int_{\mathbb{R}^d}\int_\Theta g(x)  d{P}_{\theta, h_\alpha}(x) \, d \Pi_{h_\alpha}(\theta |  x_{1:n}) \Big)\Big|_{\alpha=\alpha_0} \\
    &=\frac{\partial}{\partial \alpha} \Big(\int_{\mathbb{R}^d}  \int_\Theta g(x) f_{X_{h_\alpha}}(x|\theta) d \Pi_{h_\alpha}(\theta |  x_{1:n}) d\lambda(x)\Big)\Big|_{\alpha=\alpha_0}\\
      &= \int_{\mathbb{R}^d} g(x) \int_\Theta f_X(x|\theta)\frac{\partial}{\partial \alpha}  \Big(h'_\alpha(F_X(x|\theta)) d \Pi_{h_\alpha}(\theta |  x_{1:n})\Big)\Big|_{\alpha=\alpha_0} d\lambda(x)\\
&= \int_{\mathbb{R}^d} g(x) \int_\Theta f_X(x|\theta)\frac{\partial}{\partial \alpha}  \frac{h'_\alpha(F_X(x|\theta))\prod_{i=1}^n h'_\alpha(F_X(x_i|\theta))}{f_{X_{h_\alpha,1:n}}(x_{1:n})}\Big|_{\alpha=\alpha_0} \prod_{i=1}^nf_X(x_i|\theta) d \Pi(\theta) d\lambda(x).
    \end{align*}
where $f_{X_{h_\alpha,1:n}}(x_{1:n}) = \int_{\Theta} f_{X_{h_\alpha,1:n}}(x_{1:n}|\theta) d \Pi(\theta)$. Now note that for $h_\alpha(y)=y^\alpha$ and $\alpha_0=1$ we have
\begin{align*}
  \frac{\partial}{\partial \alpha}  f_{X_{h_\alpha,1:n}}(x_{1:n})\Big|_{\alpha=1}&= \int_\Theta\mathcal{L}(\theta)\frac{\partial}{\partial \alpha}\prod_{i=1}^n h'_\alpha(F_X(x_i|\theta))\Big|_{\alpha=1}d\Pi(\theta)\\
  &=\int_\Theta\mathcal{L}(\theta)\Big(\alpha^n \sum_{i=1}^n \log F_X(x_i |  \theta) \prod_{i=1}^n F_X(x_i |  \theta)^{\alpha-1}\Big)\Big|_{\alpha=1}d\Pi(\theta)\\
  &= \int_\Theta\mathcal{L}(\theta) \sum_{i=1}^n \log F_X(x_i |  \theta)d\Pi(\theta)=: f_{X_{1:n}}^{*}(x_{1:n})
\end{align*}
and,
\begin{align*}
   &\frac {\partial}{\partial \alpha}  \Big(h'_\alpha(F_X(x|\theta))\prod_{i=1}^nh'_\alpha(F_X(x_i|\theta))\Big)\Big|_{\alpha=1}\\
&=\Big(h'_\alpha(F_X(x|\theta))\prod_{i=1}^nh'_\alpha(F_X(x_i|\theta))\frac{\partial}{\partial \alpha} \Big(\log h'_\alpha(F_X(x|\theta)) + \sum_{i=1}^n h'_\alpha(F_X(x_i|\theta))  \Big)\Big)\Big|_{\alpha=1}\\
  &=\Big(\alpha F_X(x|\theta)^{\alpha-1}\prod_{i=1}^n\alpha F_X(x_i|\theta)^{\alpha-1}\Big(\frac{1}{\alpha}+\log F_X(x|\theta)+\frac{n}{\alpha}+\sum_{i=1}^n \log F_X(x_i|\theta)\Big)\Big)\Big|_{\alpha=1}\\
  &= 1+n+\log F_X(x|\theta)+ \sum_{i=1}^n \log F_X(x_i|\theta).
\end{align*}
As a consequence, it yields that
\begin{align*}
    &\frac{\partial}{\partial \alpha}  \frac{h'_\alpha(F_X(x|\theta))\prod_{i=1}^n h'_\alpha(F_X(x_i|\theta))}{f_{X_{h_\alpha,1:n}}(x_{1:n})}\Big|_{\alpha=1}\\
    &=\Big( \frac{
   f_{X_{h_\alpha,1:n}}(x_{1:n}) \frac {\partial}{\partial \alpha}  \Big(h'_\alpha(F_X(x|\theta))\prod_{i=1}^nh'_\alpha(F_X(x_i|\theta))\Big)
     }{f_{X_{h_\alpha,1:n}}(x_{1:n})^2}\\
    &\quad\quad - \frac{h'_\alpha(F_X(x|\theta))\prod_{i=1}^nh'_\alpha(F_X(x_i|\theta))
    \frac{\partial}{\partial \alpha }f_{X_{h_\alpha,1:n}}(x_{1:n})}{f_{X_{h_\alpha,1:n}}(x_{1:n})^2}\Big)\Big|_{\alpha=1}\\
    &=\frac{1+n+\log F_X(x|\theta) + \sum_{i=1}^n \log F_X(x_i|\theta)}{f_{X_{1:n}}(x_{1:n})}-\frac{ f_{X_{1:n}}^*(x_{1:n})}{ f_{X_{1:n}}(x_{1:n})^2},
\end{align*}
where $f_{X_{1:n}}(x_{1:n})=\int_\Theta f_{X_{1:n}} (x_{1:n}|\theta )d \Pi(\theta)$, and we used the fact that for $\alpha=1$ it holds that
$    h'_\alpha(F_X(x|\theta))=1$ and $\prod_{i=1}^n h'_\alpha(F_X(x_i|\theta))=1$. 

Note also that Assumption~\ref{assump:distortion} implies the requirement $F_X(x_i|\theta)\in (0,1]$ $\lambda \otimes \Pi$-almost surely for all $i=1,\dots,n$.
Finally, resuming the initial derivation and putting everything together, we obtain the following expression for the posterior predictive local sensitivity under power distortion
\begin{align*}
& \delta_{n,\mathrm{pred}}^{h_\alpha}= \int_{\mathbb{R}^d} g(x) \int_\Theta f_X(x|\theta)\frac{\partial}{\partial \alpha}  \frac{h'_\alpha(F_X(x|\theta))\prod_{i=1}^n h'_\alpha(F_X(x_i|\theta))}{f_{X_{h_\alpha,1:n}}(x_{1:n})}\Big|_{\alpha=\alpha_0} \prod_{i=1}^nf_X(x_i|\theta) d \Pi(\theta) d\lambda(x)\\
&=\int_{\mathbb{R}^d} g(x) \E_{\Pi^n}\Big[ f_X(x|\vartheta)f_{X_{1:n}}(x_{1:n})\frac{\partial}{\partial \alpha}  \frac{h'_\alpha(F_X(x|\vartheta))\prod_{i=1}^n h'_\alpha(F_X(x_i|\vartheta))}{f_{X_{h_\alpha,1:n}}(x_{1:n})}\Big|_{\alpha=\alpha_0}  \Big] d\lambda(x)\\
&=\int_{\mathbb{R}^d} g(x) \E_{\Pi^n}\Big[ f_X(x|\vartheta)\Big(1+n+\log F_X(x|\vartheta)+ \sum_{i=1}^n \log F_X(x_i|\vartheta)-\frac{ f_{X_{1:n}}^*(x_{1:n})}{ f_{X_{1:n}}(x_{1:n})} \Big)\Big] d\lambda(x),
\end{align*}
which is the thesis.
\end{proof}

\begin{proof}[Proof of Proposition~\ref{prop:censoring_distortion_1}]
Since ${h}'_\alpha(y) = \frac{1}{1-\alpha}$ when $y>\alpha$, we have
$$
\frac{\partial}{\partial \alpha}\log{h}'_\alpha(y)= - \frac{\partial}{\partial \alpha} \log(1-\alpha)= \frac{1}{1-\alpha},
$$
for $ y>\alpha$. Moreover, 
    \begin{align*}
     &\sum_{i=1}^n \frac{\partial}{\partial \alpha} \log  {h}'_\alpha (F_X(x_i |  \theta))\Big|_{\alpha=0}
        =\sum_{i=1}^n   \frac{\partial}{\partial \alpha} \log \Big(\frac{1}{1-\alpha} \mathbf{1}\{F_X(x_i|\theta)>\alpha\}\Big)\Big|_{\alpha=0}\\
        &=\sum_{i=1}^n \frac{\frac{1}{(1-\alpha)^2}\mathbf{1}\{F_X(x_i|\theta)>\alpha\} +\frac{1}{1-\alpha}\delta(F_X(x_i|\theta)-\alpha )  }{\frac{1}{1-\alpha} \mathbf{1}\{F_X(x_i|\theta)>\alpha\}}\Big|_{\alpha=0}\\
        &=\sum_{i=1}^n \frac{\mathbf{1}\{F_X(x_i|\theta)>0\} +\delta(F_X(x_i|\theta)-0 )  }{ \mathbf{1}\{F_X(x_i|\theta)>0\}},
    \end{align*}
where $\delta(x)$ is the Dirac delta function. The previous expression is equal to $\infty$ whenever $F_X(x_i|\theta)=0$ for some $i=1,\dots,n$. Therefore, Assumption~\ref{assump:distortion} imposes the restriction to the case where $F_X(x_i|\theta)>0$ $\Pi$-almost surely for all $i=1,\dots,n$. Under such restriction, it holds that
\begin{align*}
   \sum_{i=1}^n \frac{\partial}{\partial \alpha} \log  {h}'_\alpha (F_X(x_i |  \theta))\Big|_{\alpha=0}&=\sum_{i=1}^n \frac{\mathbf{1}\{F_X(x_i|\theta)>0 \}}{\mathbf{1}\{F_X(x_i|\theta)>0 \}}= n
\end{align*}
and, owing to Theorem~\ref{thm:lik_dist}, we obtain
\begin{align*}
   \delta^{ h_\alpha}_n  = \cov_{\Pi^n}\Big[g(\vartheta) ,n\Big]=0,
\end{align*}
which is the thesis.
\end{proof}

\begin{proof}[Proof of Proposition~\ref{prop:censoring_distortion_2}]
Since $\Tilde h'_\alpha(y)= 1/\alpha$ when $y<\alpha$, we have
$$
\frac{\partial}{\partial \alpha}\log{\Tilde h}'_\alpha(y)= -\frac{\partial}{\partial \alpha} \log \alpha  =-\frac{1}{\alpha},
$$
for $y<\alpha$. Moreover,
    \begin{align*}
        &\sum_{i=1}^n\frac{\partial}{\partial \alpha} \log  \Tilde h'_\alpha (F_X(x_i |  \theta))\Big|_{\alpha=1} =\sum_{i=1}^n    \frac{\partial}{\partial \alpha} \log \Big(\frac{1}{\alpha} \mathbf{1}\{F_X(x_i|\theta)<\alpha\}\Big)\Big|_{\alpha=1}\\
        &=\sum_{i=1}^n \frac{- \frac{1}{\alpha^2}\mathbf{1}\{F_X(x_i|\theta)<\alpha\} +\frac{1}{\alpha}\delta(\alpha-F_X(x_i|\theta))   }{\frac{1}{\alpha} \mathbf{1}\{F_X(x_i|\theta)<\alpha\}}\Big|_{\alpha=1}\\
        &=\sum_{i=1}^n \frac{-\mathbf{1}\{F_X(x_i|\theta)<1\} +\delta(1-F_X(x_i|\theta))   }{ \mathbf{1}\{F_X(x_i|\theta)< 1\}},
    \end{align*}
where $\delta(x)$ is the Dirac delta function. The previous expression is equal to $\infty$ whenever $F_X(x_i|\theta)=1$ for some $i=1,\dots,n$. Therefore, Assumption~\ref{assump:distortion} imposes the restriction to the case where $F_X(x_i|\theta)<1$ (which is equivalent to $ S_X(x_i|\theta)>0$) $\Pi$-almost surely for all $i=1,\dots,n$. Under such restriction, it holds that
\begin{align*}
\sum_{i=1}^n  \frac{\partial}{\partial \alpha} \log  \Tilde h'_\alpha (F_X(x_i |  \theta))\Big|_{\alpha=1}= -\sum_{i=1}^n\frac{\mathbf{1}\{S_X(x_i|\theta)>0 \}}{\mathbf{1}\{S_X(x_i|\theta)>0 \}} = -n
\end{align*}
and, owing to Theorem~\ref{thm:lik_dist}, we obtain
\begin{align*}
   \delta^{\Tilde h_\alpha}_n  = \cov_{\Pi^n}\Big[g(\vartheta) ,-n\Big]=0,
\end{align*}
which is the thesis.
\end{proof}

\begin{proof}[Proof of Proposition~\ref{prop:skewing_distortion}]
As observed in \cite{distorsion_priors2016}, the skewing distortion satisfies
\begin{align*}
    h'_\alpha(y)= 2F_X(\alpha F_X^{-1}(y|\theta)|\theta).
\end{align*}
It follows that
$$
\frac{\partial}{\partial \alpha} \log h'_\alpha(y) = \frac{1}{F_X(\alpha F_X^{-1}(y|\theta)|\theta)}  f_X(\alpha F_X^{-1}(y|\theta)|\theta)  F_X^{-1}(y|\theta)
$$
is well-defined at $\alpha=0$. Moreover, since 
\begin{align*}
   &\sum_{i=1}^n\frac{\partial}{\partial \alpha}  \log  {h}'_\alpha (F_X(x_i |  \theta))\Big|_{\alpha=0}= \sum_{i=1}^n \frac{\partial}{\partial \alpha }\log\Big( 2 F_X(\alpha F_X^{-1}(F_X(x_i|\theta)|\theta)|\theta)\Big)\Big|_{\alpha=0}\\
    &= \sum_{i=1}^n \frac{f_X(\alpha x_i|\theta) x_i}{F_X(\alpha x_i|\theta)}\Big|_{\alpha=0}=\frac{f_X(0|\theta)}{F_X(0|\theta)}\sum_{i=1}^n x_i =2 f_X(0|\vartheta)\sum_{i=1}^n x_i,
\end{align*}
the thesis follows directly from Theorem~\ref{thm:lik_dist}.
\end{proof}

\subsection*{Proofs for Section~\ref{sec:asymptotic_behaviour}}
\begin{proof}[Proof of Theorem~\ref{thm:asymp}]
 For $i=1,\dots,n$ let us consider the function
    $$\vartheta \mapsto r_i(\vartheta):= \frac{\partial}{\partial \alpha}\log  (h'_\alpha(F_{X}(X_i|\vartheta)))\Big|_{\alpha=\alpha_0},$$
    thus, we consider 
$$
 \delta^{h_\alpha}_n=\cov_{\Pi^n}\Big[g(\vartheta), \sum_{i=1}^nr_i(\vartheta)\Big],
$$
    where $r_i:\Theta\to \mathbb{R}$ is a function of $\vartheta$ satisfying Assumption~\ref{assump:ftheta} for all $i=1,\dots,n$. Recall also that $g:\Theta \to \mathbb{R}^s$ and $\dim(\Theta)=k$.
    Combining the assumptions of the theorem with Taylor's Theorem, we can expand $r_i$ and $g$ around $\theta_0$ obtaining
    \begin{align*}
    g(\vartheta)=g(\theta_0)+
     J_{\scriptscriptstyle \theta} g(\theta_0)(\vartheta-\theta_0)+ o(d(\vartheta,\theta_0))
    \end{align*}
     where $J_{\scriptscriptstyle \theta} g$ is the $s\times k$ Jacobian matrix of $g$ w.r.t.\ $\theta$, and 
    \begin{align*}
      r_i(\vartheta)=r_i(\theta_0)+\nabla_{\scriptscriptstyle \theta} r_i(\theta_0)^\top(\vartheta-\theta_0)+ o(d(\vartheta,\theta_0))
    \end{align*}
     where $\nabla_{\scriptscriptstyle \theta} r_i$ is the gradient of $r_i$ w.r.t.\ $\vartheta$.
    By assumption the posterior $\Pi^n$ is consistent at $\theta_0$ with contraction rate $\varepsilon_n$ in the semi-metric $d$ under $P^n_{\theta_0}$. In other words,
     $$
   P_{\theta_0}^n \left( d(\vartheta, \theta_0) = O_{\Pi_n}(M_n\varepsilon_n) \right) \to 1
     $$
where $M_n\to\infty$ can be taken arbitrarily slow. With a small abuse of notation, we omit the additional term $P_{\theta_0}^n$ that should be included in the remainder of the Taylor series and we obtain
     \begin{align*}
    g(\vartheta)&=g(\theta_0)+J_{\scriptscriptstyle \theta} g(\theta_0)(\vartheta-\theta_0)+o_{\Pi_n}(M_n\varepsilon_n), \\
    r_i(\vartheta)&=r_i(\theta_0)+\nabla_{\scriptscriptstyle \theta} r_i(\theta_0)^\top(\vartheta-\theta_0)+o_{\Pi_n}(M_n\varepsilon_n).
     \end{align*}
By Assumption~\ref{assump:ftheta}, $J_{\scriptscriptstyle \theta} g(\theta_0) = O(1)$ and
$ \sum_{i=1}^n \nabla_{\scriptscriptstyle \theta} r_i(\theta_0) = O(n)$, thus, using the properties of the covariance operator we have
     \begin{align*}
       \delta^{ h_\alpha}_n&= \cov_{\Pi^n}\Big[g(\vartheta),  \sum_{i=1}^nr_i(\vartheta)\Big]\\
       &= \cov_{\Pi^n}\Big[
     J_{\scriptscriptstyle \theta} g(\theta_0)(\vartheta-\theta_0)+o_{\Pi_n}(M_n\varepsilon_n),  \sum_{i=1}^n \nabla_{\scriptscriptstyle \theta} r_i(\theta_0)^\top (\vartheta-\theta_0) + \sum_{i=1}^n o_{\Pi_n}(M_n\varepsilon_n)\Big] \\
     &=\sum_{i=1}^n  \cov_{\Pi^n}\Big[
     J_{\scriptscriptstyle \theta} g(\theta_0)(\vartheta-\theta_0)+o_{\Pi_n}(M_n\varepsilon_n),   \nabla_{\scriptscriptstyle \theta} r_i(\theta_0)^\top (\vartheta-\theta_0) +  o_{\Pi_n}(M_n\varepsilon_n)\Big] \\
    &=  J_{\scriptscriptstyle \theta} g(\theta_0) \cov_{\Pi^n}[\vartheta,\vartheta]\sum_{i=1}^n \nabla_{\scriptscriptstyle \theta} r_i(\theta_0) +o_{\scriptscriptstyle P^{n}_{\theta_0}}(n M_n^2\varepsilon_n^2)\\
   & = O_{\scriptscriptstyle P^{n}_{\theta_0}}(n M_n^2\varepsilon_n^2),
     \end{align*}
where we exploited the fact that $\E_{\Pi^n}[\vartheta-\theta_0]=O{\scriptscriptstyle P^{n}_{\theta_0}}(M_n\varepsilon_n)$, namely, the posterior mean achieves contraction rate $\varepsilon_n$.
This is a consequence of Theorem~8.8 of \cite{ghosal2017fundamentals} and Assumption~\ref{assump:metric} which imply that
\begin{align*}
d (\E_{\Pi^n}[\vartheta], \theta_0)&\leq M_n \varepsilon_n + \|d\|_\infty^{1/q}\Pi^n(\theta: d(\theta,\theta_0)>M_n\varepsilon_n)^{1/q} =  O{\scriptscriptstyle P^{n}_{\theta_0}}(M_n\varepsilon_n),
\end{align*}
for $q\geq 1$ defined in Assumption~\ref{assump:metric}. To conclude, it suffices to get rid of $M_n$ using its arbitrariness and by observing that, for instance, the sequence $d (\E_{\Pi^n}[\vartheta],\theta_0)/\varepsilon_n$ is tight.
Suppose, by contradiction, that $d(E_{\Pi_n}[\vartheta],\theta_0)/\varepsilon_n$ is not tight. Then there exist $K>0$ and $\delta>0$ such that
$$P_{\theta_0}^n(d(E_{\Pi_n}[\vartheta],\theta_0)/\varepsilon_n > K) > \delta$$
infinitely often.
Now, by arbitrariness of $M_n$, define a diverging $M_n$ equal to $K/C$ on the infinitely many indices where the probability above exceeds $\delta$. On these indices, equality $C M_n = K$ gives
$P_{\theta_0}^n(d(E_{\Pi_n}[\vartheta],\theta_0)/\varepsilon_n > C M_n) > \delta$
infinitely often, contradicting posterior contraction. Hence $d(E_{\Pi_n}[\vartheta],\theta_0)/\varepsilon_n$ is tight and the claim follows.
\end{proof}

\begin{proof}[Proof of Theorem~\ref{thm:clt}]
   By the same arguments and notation of proof of Theorem~\ref{thm:asymp} for $\varepsilon_n=n^{-1/2}$ we have
   \begin{align*}
         \delta^{h_\alpha}_n  \sim  J_{\scriptscriptstyle \theta} g(\theta_0) \cov_{\Pi^n}[\vartheta,\vartheta]\sum_{i=1}^n \nabla_{\scriptscriptstyle \theta} r_i(\theta_0)=  \frac{C_n}{n} \sum_{i=1}^n \nabla_{\scriptscriptstyle \theta} r_i(\theta_0),
   \end{align*}
    where 
    $
    C_n= n J_{\scriptscriptstyle \theta} g(\theta_0)\cov_{\Pi^n}[\vartheta,\vartheta]
    $
    is an $s \times k$ random matrix.
    Now it is not difficult to see that the arguments of the proof of Theorem~\ref{thm:asymp} together with Assumption~\ref{assump:ftheta} imply that $\{C_n\}_{n \in \mathbb{N}}$ is a tight sequence of random matrices (in $P^{n}_{\theta_0}$-probability), i.e.\ $C_n=O{\scriptscriptstyle P^{n}_{\theta_0}}(1)$. Moreover, owing to Theorem~\ref{thm:asymp}, it follows that $\delta_n^{h_\alpha}$ is also bounded in $P^{n}_{\theta_0}$-probability. Thus, let $\delta_0\in\mathbb{R}^s$, centring and re-scaling we obtain
    \begin{align*}
        W_n&= \sqrt{n}(\delta_n ^{h_\alpha}- \delta_0)\\
        &\sim \sqrt{n} \Big(J_{\scriptscriptstyle \theta} g(\theta_0) \cov_{\Pi^n}[\vartheta,\vartheta]\sum_{i=1}^n \nabla_{\scriptscriptstyle \theta} r_i(\theta_0) -\delta_0\Big)\\
         &= \sqrt{n} \Big(\frac{C_n}{n} \sum_{i=1}^n \nabla_{\scriptscriptstyle \theta} r_i(\theta_0)-\delta_0\Big)\\
         &= \frac{C_n}{\sqrt{n}} \sum_{i=1}^n \nabla_{\scriptscriptstyle \theta} r_i(\theta_0)-\sqrt{n} \delta_0\\
        &= \frac{C_n}{\sqrt{n}}\Big(\sum_{i=1}^n \nabla_{\scriptscriptstyle \theta} r_i(\theta_0) -  \frac{n\delta_0}{C_n}\Big).
    \end{align*}
Therefore, since $r_i(\vartheta)= \frac{\partial}{\partial \alpha}\log  (h'_\alpha(F_{X}(X_i|\vartheta)))|_{\alpha=\alpha_0}$ and $n ^{-1/2}\sum_{i=1}^n X_i$ is asymptotically Gaussian, we invoke the Central Limit Theorem to conclude that there exist $\delta_0\in \mathbb{R}^s$ and $\Sigma\in  \mathbb{R}^{s\times s}$ appropriately chosen, such that $Z_n=\Sigma^{-1/2}W_n$ converges in distribution to a standard multivariate Gaussian r.v. $Z\sim \mathcal{N}(0,I)$.
\end{proof}
\bibliographystyle{apalike}
\bibliography{bib.bib}

\end{document}